\documentclass[12pt]{amsart}

\usepackage[colorlinks=true]{hyperref}
\usepackage{amssymb}
\usepackage[all]{xy}
\usepackage{mathtools}
\usepackage{stackrel}
\usepackage{verbatim}
\usepackage{eucal}

\usepackage{tikz}
\usetikzlibrary{calc,matrix}
\tikzset{w/.style={circle, draw,inner sep=1pt},b/.style={circle,draw,fill,inner sep=2pt}, s/.style={rectangle, draw,inner sep=3pt}}

\newcommand{\nocontentsline}[3]{}
\newcommand{\tocless}[2]{\bgroup\let\addcontentsline=\nocontentsline#1{#2}\egroup}

\newcommand{\bE}{\mathbb{E}}
\newcommand{\G}{\mathbf{G}}
\newcommand{\bL}{\mathbb{L}}

\newcommand{\bP}{\mathbb{P}}
\newcommand{\s}{\mathbf{s}}
\newcommand{\bT}{\mathbb{T}}
\renewcommand{\t}{\mathbf{t}}

\newcommand{\Z}{\mathbb{Z}}
\newcommand{\bZ}{\mathbf{Z}}

\newcommand{\cA}{\mathcal{A}}
\newcommand{\cB}{\mathcal{B}}
\newcommand{\cC}{\mathcal{C}}

\newcommand{\cI}{\mathcal{I}}
\newcommand{\cL}{\mathcal{L}}
\newcommand{\cM}{\mathcal{M}}
\newcommand{\cO}{\mathcal{O}}
\newcommand{\cP}{\mathcal{P}}
\newcommand{\cV}{\mathcal{V}}

\newcommand{\C}{\mathrm{C}}
\newcommand{\rH}{\mathrm{H}}
\newcommand{\rN}{\mathrm{N}}
\newcommand{\T}{\mathrm{T}}
\newcommand{\U}{\mathrm{U}}
\newcommand{\rZ}{\mathrm{Z}}

\newcommand{\g}{\mathfrak{g}}

\newcommand{\Arr}{\mathrm{Arr}}
\newcommand{\bu}{\mathbf{1}}
\newcommand{\Br}{\mathrm{Br}}
\newcommand{\coeq}{\mathrm{coeq}}
\newcommand{\Cois}{\mathrm{Cois}}
\newcommand{\Conv}{\mathrm{Conv}}
\newcommand{\cu}{\mathrm{cu}}
\newcommand{\Def}{\mathrm{Def}}
\newcommand{\Der}{\mathrm{Der}}
\newcommand{\dg}{\mathrm{dg}}
\newcommand{\bDR}{\mathbf{DR}}
\newcommand{\End}{\mathrm{End}}
\newcommand{\Free}{\mathrm{Free}}

\newcommand{\Harr}{\mathrm{Harr}}
\newcommand{\Hom}{\mathrm{Hom}}
\newcommand{\id}{\mathrm{id}}
\newcommand{\Isom}{\mathrm{Isom}}
\newcommand{\Map}{\mathrm{Map}}
\newcommand{\MC}{\mathrm{MC}}
\renewcommand{\nu}{\mathrm{nu}}
\newcommand{\Op}{\mathrm{Op}}
\newcommand{\opp}{\mathrm{opp}}
\newcommand{\Pois}{\mathrm{Pois}}
\newcommand{\Pol}{\mathrm{Pol}}
\newcommand{\bPol}{\mathbf{Pol}}
\newcommand{\SC}{\mathrm{SC}}
\newcommand{\sgn}{\mathrm{sgn}}
\newcommand{\Sym}{\mathrm{Sym}}
\newcommand{\Tree}{\mathrm{Tree}}
\newcommand{\triv}{\mathrm{triv}}

\newcommand{\Ass}{\mathrm{Ass}}
\newcommand{\coAss}{\mathrm{coAss}}
\newcommand{\Comm}{\mathrm{Comm}}
\newcommand{\coComm}{\mathrm{coComm}}
\newcommand{\coLie}{\mathrm{coLie}}
\newcommand{\Lie}{\mathrm{Lie}}
\newcommand{\preLie}{\mathrm{preLie}}
\newcommand{\coP}{\mathrm{co}\mathbb{P}}

\newcommand{\alg}{\mathrm{Alg}}
\newcommand{\balg}{\mathbf{Alg}}
\newcommand{\CAlg}{\mathrm{CAlg}}
\newcommand{\bCAlg}{\mathbf{CAlg}}
\newcommand{\comod}{\mathrm{CoMod}}
\renewcommand{\mod}{\mathrm{Mod}}

\newcommand{\blmod}{\mathbf{LMod}}
\newcommand{\brmod}{\mathbf{RMod}}
\newcommand{\bimod}[2]{{}_{#1}\mathbf{BMod}_{#2}}

\DeclareMathOperator{\colim}{colim}
\newcommand{\iHom}{\underline{\mathrm{Hom}}}

\renewcommand{\d}{\mathrm{d}}
\newcommand{\ddr}{\mathrm{d}_{\mathrm{dR}}}

\newtheorem{thm}{Theorem}[section]
\newtheorem*{theorem}{Theorem}
\newtheorem{prop}[thm]{Proposition}
\newtheorem{lm}[thm]{Lemma}
\newtheorem{cor}[thm]{Corollary}

\theoremstyle{definition}
\newtheorem{defn}[thm]{Definition}

\theoremstyle{remark}
\newtheorem{remark}[thm]{Remark}

\begin{document}
\title{Derived coisotropic structures I: affine case}

\address{Dipartimento di Matematica, Universit\`a di Milano, Via Cesare Saldini 50, 20133 Milan, Italy}
\email{valerio.melani@outlook.com}
\author{Valerio Melani}

\address{Department of Mathematics, University of Geneva, 2-4 rue du Lievre, 1211 Geneva, Switzerland}
\curraddr{Institut f\"{u}r Mathematik, Winterthurerstrasse 190, 8057 Z\"{u}rich, Switzerland}
\email{pavel.safronov@math.uzh.ch}
\author{Pavel Safronov}

\begin{abstract}
We define and study coisotropic structures on morphisms of commutative dg algebras in the context of shifted Poisson geometry, i.e. $\bP_n$-algebras. Roughly speaking, a coisotropic morphism is given by a $\bP_{n+1}$-algebra acting on a $\bP_n$-algebra. One of our main results is an identification of the space of such coisotropic structures with the space of Maurer--Cartan elements in a certain dg Lie algebra of relative polyvector fields. To achieve this goal, we construct a cofibrant replacement of the operad controlling coisotropic morphisms by analogy with the Swiss-cheese operad which can be of independent interest. Finally, we show that morphisms of shifted Poisson algebras are identified with coisotropic structures on their graph.
\end{abstract}
\maketitle

\tableofcontents

\addtocontents{toc}{\protect\setcounter{tocdepth}{1}}

\section*{Introduction}

This paper is a continuation of the work of Calaque, Pantev, To\"{e}n, Vaqui\'{e} and Vezzosi \cite{CPTVV} on shifted Poisson structures on derived stacks. In this paper we introduce shifted coisotropic structures on morphisms of commutative dg algebras and show that they possess certain expected properties:
\begin{itemize}
\item Suppose $A$ is a $\bP_{n+1}$-algebra, i.e. a commutative dg algebra with a Poisson bracket of degree $-n$. Then the identity morphism $A\rightarrow A$ carries a unique $n$-shifted coisotropic structure.

\item A morphism of commutative algebras $A\rightarrow B$ is compatible with $n$-shifted Poisson structures iff its graph $A\otimes B\rightarrow B$ has an $n$-shifted coisotropic structure.
\end{itemize}

We generalize these results to general derived Artin stacks in \cite{MS2}.

\subsection*{Classical setting}

Let us recall two ways of defining Poisson structures and coisotropic embeddings in the classical setting. Suppose $X$ is a smooth scheme over a characteristic zero field $k$. Then one has the following equivalent definitions of a Poisson structure on $X$:
\begin{enumerate}
\item The structure sheaf $\cO_X$ of $X$ is a sheaf of $k$-linear Poisson algebras where the multiplication coincides with the original commutative multiplication on $\cO_X$.

\item $X$ carries a bivector $\pi_X\in\rH^0(X, \wedge^2 \T_X)$ such that $[\pi_X, \pi_X] = 0$.
\end{enumerate}

The equivalence of the two definitions is clear: a bivector $\pi_X$ is the same as an antisymmetric biderivation $\cO_X\otimes_k \cO_X\rightarrow \cO_X$; an easy computation shows that the equation $[\pi_X, \pi_X] = 0$ is then equivalent to the Jacobi identity for the corresponding biderivation.

Now suppose $X$ is a scheme carrying a Poisson structure in one of the senses above and consider a smooth closed subscheme $i\colon L\hookrightarrow X$. Then one has the following equivalent definitions of $L$ being coisotropic:
\begin{enumerate}
\item The ideal sheaf $\cI_L$ defining $L$ is a coisotropic ideal, i.e. it is closed under the Poisson bracket on $\cO_X$.

\item Let $\rN_{L/X}$ be the normal bundle of $L$. The composite
\[\rN^*_{L/X}\longrightarrow i^* \T^*_X\stackrel{\pi_X}\longrightarrow i^*\T_X\longrightarrow \rN_{L/X}\]
is zero.
\end{enumerate}

The equivalence of the two definitions is well-known and follows from the identification $\rN^*_{L/X}\cong \cI_L / \cI_L^2$.

\subsection*{Shifted Poisson algebras}

Now suppose $A$ is a commutative dg algebra. To generalize the first definition of a Poisson structure we simply replace the Poisson bracket of degree $0$ by one of degree $-n$.

To generalize the second definition, consider
\[\bPol(A, n) = \Hom_A(\Sym_A(\bL_A[n+1]), A),\]
the algebra of $n$-shifted polyvector fields, where $\bL_A$ is the cotangent complex of $A$. It is equipped with a commutative multiplication, a Poisson bracket of degree $-n-1$ generalizing the Schouten bracket and an additional grading such that $\bL_A$ has weight $-1$, i.e. $\bPol(A, n)$ is a graded $\bP_{n+2}$-algebra. A bivector of the correct degree is a morphism of graded complexes $k(2)[-1]\rightarrow \bPol(A, n)[n+1]$ and the Jacobi identity can be succinctly summarized by assuming that the morphism $k(2)[-1]\rightarrow \bPol(A, n)[n+1]$ is a map of graded dg Lie algebras.

Therefore, we can give the following two definitions of an $n$-shifted Poisson structure on a cdga $A$:
\begin{enumerate}
\item $A$ carries a $\bP_{n+1}$-algebra structure compatible with the original commutative structure on $A$.

\item One has a morphism of graded dg Lie algebras
\[k(2)[-1]\longrightarrow \bPol(A, n)[n+1].\]
\end{enumerate}

In fact, both definitions give a whole \emph{space} $\Pois(A, n)$ of $n$-shifted Poisson structures. An equivalence of the two spaces is the main theorem of \cite{Me} (see also Theorem \ref{thm:poissonpolyvectors}).

\subsection*{Shifted coisotropic structures}

Now suppose $f\colon A\rightarrow B$ is a morphism of commutative algebras. One can generalize the classical definitions of coisotropic submanifolds to $n$-shifted coisotropic structures as follows.

To generalize the first definition, let us consider a certain colored dg operad $\bP_{[n+1, n]}$ introduced in \cite{Sa1}. Its algebras are triples $(A, B, F)$ consisting of a $\bP_{n+1}$-algebra $A$, a $\bP_n$-algebra $B$ and a morphism of $\bP_{n+1}$-algebras $F\colon A\rightarrow \rZ(B)$. Here $\rZ(B)$ is the so-called Poisson center and is given by twisting the differenital of $\bPol(B, n-1)$ by $[\pi_B, -]$. In particular, we obtain a morphism of commutative algebras $A\rightarrow \rZ(B)\rightarrow B$. To establish a relation between this definition and the classical definition of coisotropic ideals, we show in Section \ref{sect:relpoistopois} that the homotopy fiber $\U(A, B)$ of the morphism $A\rightarrow B$ carries a canonical structure of non-unital $\bP_{n+1}$-algebra such that the projection $\U(A, B)\rightarrow A$ is a morphism of $\bP_{n+1}$-algebras.

To generalize the second definition, consider the algebra of $n$-shifted relative polyvectors
\[\bPol(B/A, n-1) = \Hom_B(\Sym_B(\bL_{B/A}[n-1]), B).\]
It is equipped with a graded $\bP_{n+1}$-algebra structure as before. There is a natural morphism of commutative algebras
\[\bPol(A, n)\longrightarrow \bPol(B/A, n-1)\]
induced from the morphism $\bL_{B/A}\rightarrow f^*\bL_A[1]$ of cotangent complexes. In Section \ref{sect:relativepoly} we show that the pair $(\bPol(A, n), \bPol(B/A, n-1))$ is naturally a graded $\bP_{[n+2, n+1]}$-algebra and thus we obtain a graded non-unital $\bP_{n+2}$-algebra $\bPol(f, n) = \U(\bPol(A, n), \bPol(B/A, n-1))$. Thus, we can consider morphisms of graded dg Lie algebras
\[k(2)[-1]\longrightarrow \bPol(f, n)[n+1].\]
To see that this indeed generalizes the classical definition, consider the case when $f$ is surjective which corresponds to a closed embedding of a subscheme. Then the morphism $\bPol(A, n)\rightarrow \bPol(B/A, n-1)$ is surjective as well and by Proposition \ref{prop:strictrelpolyvectors} we can identify $\bPol(f, n)$ with its strict kernel. Thus, bivectors in $\bPol(f, n)$ are bivectors in $\bPol(A, n)$ which vanish when restricted to the normal bundle of $B$.

Therefore, we can make the following two definitions of an $n$-shifted coisotropic structure on a cdga morphism $f\colon A\rightarrow B$:
\begin{enumerate}
\item One has a $\bP_{[n+1, n]}$-algebra structure on $(A, B)$ such that the induced morphism $A\rightarrow \rZ(B)\rightarrow B$ of commutative algebras is homotopic to $f$.

\item One has a morphism of graded dg Lie algebras
\[k(2)[-1]\longrightarrow \bPol(f, n)[n+1].\]
\end{enumerate}

Let us denote the space of $n$-shifted coisotropic structures in the first sense by $\Cois(f, n)$. The following theorem is the first main result of the paper (Theorem \ref{thm:coisotropicpolyvectors}):
\begin{theorem}
Let $f\colon A\rightarrow B$ be a morphism of commutative dg algebras. Then one has an equivalence of spaces
\[\Cois(f, n)\cong \Map_{\alg_{Lie}^{gr}}(k(2)[-1], \bPol(f, n)[n+1]).\]
\end{theorem}

Let us mention that our first definition was based on \cite[Definition 3.4.4]{CPTVV}, which is given as follows. By a theorem proved independently by Rozenblyum and the second author (see \cite[Theorem 2.22]{Sa2}) there is an equivalence of $\infty$-categories
\begin{equation}
\balg_{\bP_{n+1}}\stackrel{\sim}\longrightarrow \balg(\balg_{\bP_n})
\label{eq:Poissonadditivity}
\end{equation}
beween $\bP_{n+1}$-algebras and associative algebra objects in $\bP_n$-algebras. Thus, one can define an $n$-shifted coisotropic structure in terms of an associative action of the $\bP_{n+1}$-algebra $A$ on the $\bP_n$-algebra $B$. An equivalence of the two definitions is shown in \cite[Corollary 3.8]{Sa2}.

\subsection*{Braces and the Swiss-cheese operad}

Let us briefly explain some intermediate constructions that go into the proof of Theorem \ref{thm:coisotropicpolyvectors} which can be of independent interest.

To prove the claim, we have to construct a cofibrant replacement of the colored operad $\bP_{[n+1, n]}$, which is done as follows. First of all, we have to replace the center $\rZ(B)$ by a homotopy center $\bZ(B)$ which is defined as an operadic convolution algebra associated to the cooperad $\coP_n$ Koszul dual to $\bP_n$. A priori the homotopy center $\bZ(B)$ associated to some cooperad $\cC$ is merely a Lie algebra, but as shown by Calaque and Willwacher \cite{CW}, it carries a canonical action of the so-called brace operad $\Br_{\cC}$ if $\cC$ is a Hopf cooperad, i.e. a cooperad in commutative algebras, which is the case for $\cC=\coP_n$. Moreover, they have constructed a morphism from a resolution of $\bP_{n+1}$ to $\Br_{\coP_n}$. Thus, if $B$ is a $\bP_n$-algebra, $\bZ(B)$ becomes a homotopy $\bP_{n+1}$-algebra.

We can now generalize the problem of construction of a cofibrant replacement for $\bP_{[n+1, n]}$ as follows. Let $\cC_1$ be a dg cooperad and $\cC_2$ a Hopf dg cooperad together with a morphism $\Omega(\cC_1)\rightarrow \Br_{\cC_2}$. In Section \ref{sect:SCalgebras} we construct a certain semi-free colored operad $\SC(\cC_1, \cC_2)$ whose algebras are given by triples $(A, B, F)$, where $A$ is an $\Omega(\cC_1)$-algebra, $B$ is an $\Omega(\cC_2)$-algebra and $F\colon A\rightarrow \bZ(B)$ is an $\infty$-morphism of $\Omega(\cC_1)$-algebras. In the case of Poisson algebras we get a cofibrant colored operad $\widetilde{\bP}_{[n+1, n]}$ whose algebras are triples $(A, B, F)$ consisting of a homotopy $\bP_{n+1}$-algebra $A$, homotopy $\bP_n$-algebra $B$ and an $\infty$-morphism of homotopy $\bP_{n+1}$-algebras $F\colon A\rightarrow \bZ(B)$. Moreover, by Proposition \ref{prop:relPnresolution} the natural morphism $\widetilde{\bP}_{[n+1, n]}\rightarrow \bP_{[n+1, n]}$ is a weak equivalence.

Let us mention why we call $\SC(\cC_1, \cC_2)$ the Swiss-cheese construction. Consider the colored operad $\SC = \SC(\mathrm{co}\bE_2, \coAss)$. Its algebras are triples $(A, B, F)$ consisting of an $\bE_2$-algebra $A$, an $\bE_1$-algebra $B$ and a morphism of $\bE_2$-algebras $A\rightarrow \mathrm{HH}^\bullet(B, B)$, where $\mathrm{HH}^\bullet(B, B)$ is the Hochschild cochain complex endowed with its natural $\bE_2$-structure. It is thus expected that $\SC$ is a model for the chain operad on the two-dimensional Swiss-cheese operad of Voronov (we refer to \cite{Th} for a similar statement in the topological setting).

We also generalize the construction of Calaque--Willwacher of the brace operad to a relative setting of a morphism of $\Omega\cC$-algebras $A\rightarrow B$. In this case one can consider a triple of a convolution algebra of $A$, a convolution algebra of $B$ and a relative convolution algebra. For instance, one can interpret $\bPol(B/A, n-1)$ as a relative convolution algebra of the morphism of graded $\bP_{n+1}$-algebras $A\rightarrow B\rightarrow \bPol(B, n-1)$, where both $A$ and $B$ are equipped with the zero Poisson bracket. Therefore, the construction of the graded $\bP_{[n+2, n+1]}$-algebra on the pair $(\bPol(A, n), \bPol(B/A, n-1))$ follows from the action of the convolution algebra of $A$ on the relative convolution algebra constructed in Proposition \ref{prop:bracemorphisms}.

\subsection*{Poisson morphisms}

Let us finish the introduction by stating the last main result of the paper. Recall that a morphism of Poisson manifolds $Y\rightarrow X$ is Poisson iff its graph $Y\rightarrow \overline{Y}\times X$ is coisotropic, where $\overline{Y}$ is the same manifold with the opposite Poisson structure. We show that a similar statement holds in the derived context as well. Namely, we have the following statement (see Theorem \ref{thm:graphcoisotropic}).

\begin{theorem}
Let $A,B$ be $\bP_{n+1}$-algebras, and let $f\colon A\rightarrow B$ be a morphism of commutative algebras. Then the space $\Pois(f, n)$ of lifts of $f$ to a morphism of $\bP_{n+1}$-algebras is equivalent to the space of 
coisotropic structures on $A \otimes \overline{B} \to B$, where $\overline{B}$ is the same algebra $B$ with the opposite $\bP_{n+1}$-structure.
\end{theorem}

Here by a decomposable $n$-shifted coisotropic structure on $A\otimes B$ we mean one whose underlying $n$-shifted Poisson structure on $A\otimes B$ is obtained by a sum of $n$-shifted Poisson structures on $A$ and $B$ separately.

The main idea of the proof is to use the Poisson additivity functor \eqref{eq:Poissonadditivity} to reduce the statement to the following basic algebraic fact (see Lemma \ref{lm:morphismbimodule}). Let $A$ and $B$ be associative algebras and consider the right $B$-module $B_B$. Then lifts of $B_B$ to an $(A, B)$-bimodule are equivalent to specifying a morphism of associative algebras $A\rightarrow B$ since an $(A, B)$-bimodule ${}_AB_B$ gives rise to a morphism of associative algebras $A\rightarrow \End_{\brmod_B}(B)\cong B$.

\subsection*{Acknowledgements}

We would like to thank D. Calaque, G. Ginot, M. Porta, B. To\"{e}n and G. Vezzosi for many interesting and stimulating discussions. The work of P.S. was supported by the EPSRC grant EP/I033343/1. We also thank the anonymous referee for his/her useful comments.

\section{Basic definitions}
\addtocontents{toc}{\protect\setcounter{tocdepth}{2}}

\subsection{Model categories}
\label{sect:modelcats}

Let $k$ be a field of characteristic zero and let $\dg_k$ be the symmetric monoidal category of unbounded chain complexes of $k$-vector spaces. Moreover, it carries a projective model structure, whose weak equivalences are the quasi-isomorphisms and whose fibrations are degree-wise surjections. Together these structures are compatible in the sense that $\dg_k$ forms a symmetric monoidal model category.

Throughout the paper we are going to work in the setting of a general model category $M$ following \cite[Section 1.1]{CPTVV}, which the reader can safely assume to be the category of chain complexes $\dg_k$ or the category of diagrams in it. In this section we will list the necessary assumptions on the model category $M$.

Consider $M$, a closed symmetric monoidal combinatorial model category. In addiction to this, suppose $M$ is enriched over $\dg_k$, or, equivalently, $M$ is a symmetric monoidal $\dg_k$-model algebra in the sense of Hovey (see \cite[Definition 4.2.20]{Ho}). It is shown in \cite[Appendix A.1.1]{CPTVV} that such an $M$ is a stable model category. Furthermore, we will make the following assumptions on $M$:
\begin{enumerate}
\item The unit $\bu_M$ of $M$ is cofibrant.

\item Suppose $f\colon A \rightarrow B$ is a cofibration and $C$ an object of $M$. Then for any morphism $A\otimes C \rightarrow D$ the strict pushout of the diagram
\[
\xymatrix{
A\otimes C \ar[r] \ar[d] & D \\
B\otimes C &
}
\]
is also a model for the homotopy pushout.

\item If $A$ is a cofibrant object, then the functor $A\otimes -\colon M\rightarrow M$ preserves weak equivalences.

\item $M$ is a tractable model category.

\item Finite products and filtered colimits preserve weak equivalences.
\end{enumerate}

We denote by $\cM$ the localization of the model category $M$ which is a symmetric monoidal dg category.

\subsection{Graded mixed objects}
\label{sect:gradedmixed}

Most of the results in this section can be found in \cite[Section 1.1]{CPTVV}, so we will be brief; the reader is invited to consult the reference for details.

Let $V$ be a complex and $\epsilon$ a square-zero endomorphism of $V$ of degree 1 such that $[\d, \epsilon] = 0$. We can then twist the differential on $V$ by $\epsilon$, i.e. consider the same underlying graded vector space equipped with the differential $\d+\epsilon$. This construction does not preserve quasi-isomorphisms and so does not make sense in the underlying $\infty$-category. We use the realization functors of graded mixed objects as replacements for this construction: these preserve weak equivalences and make sense for any $\cM$.

Let us recall the symmetric monoidal model categories $M^{gr}$ and $M^{gr, \epsilon}$ of graded and graded mixed objects of $M$ respectively. We denote by $\cM^{gr}$ and $\cM^{gr, \epsilon}$ the corresponding $\infty$-categories.

\begin{defn}
The category of \emph{graded objects} in $M$
\[M^{gr}=\comod_{\cO(\G_m)}(M)\]
is the category of comodules over $\cO(\G_m)\cong k[t, t^{-1}]$ with $\deg(t) = 0$.
\end{defn}
Explicitly, an object of $M^{gr}$ consists of a collection $\{A(n)\}_{n\in \Z}$ of objects of $M$ with tensor product defined by
\[(A\otimes B)(n) = \bigoplus_{n_1+n_2=n} A(n_1)\otimes B(n_2).\]
Note that the braiding isomorphism does not involve Koszul signs with respect to this grading. We will refer to this grading as the \emph{weight} grading.

Now consider the commutative bialgebra $B = k[x, t, t^{-1}]$ with $\deg(x) = -1$, $\deg(t) = 0$, where $\Delta(t) = t\otimes t$ and $\Delta(x) = x\otimes 1 + t\otimes x$.

\begin{defn}
The category of \emph{graded mixed objects} in $M$
\[M^{gr, \epsilon}=\comod_B(M)\]
is the category of comodules over $B$.
\end{defn}
Explicitly, objects of $M^{gr, \epsilon}$ are graded objects $\{A(n)\}_{n\in \Z}$ together with operations \[\epsilon\colon A(n)\rightarrow A(n+1)[1]\] such that $\epsilon^2=0$. Equivalently, $M^{gr, \epsilon}$ is the category of comodules over $k[x]$ in $M^{gr}$, where the weight and degree of $x$ are both $-1$. Since $k[x]$ is dualizable, $M^{gr, \epsilon}$ is equivalently the category of modules over $k[\epsilon]$ in $M^{gr}$, where the weight and degree of $\epsilon$ are both $1$.

The category $M^{gr, \epsilon}$ has a projective model structure whose weak equivalences and fibrations are defined componentwise.

Consider the functor $\triv\colon \cM\rightarrow \cM^{gr, \epsilon}$ which associates to any object of $\cM$ the same object with the trivial graded mixed structure. Let $\bu_M(0)$ be the trivial graded mixed object concentrated in weight $0$. Then $\triv(V) = V\otimes \bu_M(0)$. It is naturally a symmetric monoidal functor.
\begin{defn}
The \emph{realization functor}
\[|-|\colon \cM^{gr, \epsilon}\longrightarrow \cM\]
is the right adjoint to the trivial functor $\triv\colon \cM\rightarrow \cM^{gr, \epsilon}$.
\end{defn}
Explicitly, $|A| = \underline{\Map}_{\cM^{gr, \epsilon}}(\bu_M(0), A)$, where $\underline{\Map}_{\cM^{gr, \epsilon}}(-, -)$ is the $\cM$-enriched Hom. The realization functor has the following strict model. $\bu_M(0)$ has a cofibrant replacement in the projective model structure given by $\tilde{k}\otimes \bu_M$ for $\tilde{k}=k[z,w]$, where $\deg(z)=0$, $\deg(w) = 1$ and the weights of both $z$ and $w$ are 1. We define $\d z = w$ and $\epsilon z = wz$. The natural morphism $\tilde{k}\rightarrow k$ given by setting $z=w=0$ is a weak equivalence. Assume $A\in M^{gr, \epsilon}$ is fibrant in the projective model structure. Then
\[|A| = \underline{\Map}_{M^{gr, \epsilon}}(\bu_M(0), A) \cong \iHom_{M^{gr, \epsilon}}(1_M\otimes \tilde{k}, A) \in M.\]

Post-composing the realization functor with the forgetful functor $M\rightarrow \dg_k$ we get the following description:
\[|A|\cong \prod_{n\geq 0} A(n)\]
with the differential given by twisting the original differential by $\epsilon$. Since $|-|$ is a right adjoint to a symmetric monoidal functor, it naturally has a structure of a lax monoidal functor.

\begin{defn}
The \emph{left realization} functor
\[|-|^l\colon \cM^{gr, \epsilon}\rightarrow \cM\]
is the left adjoint to the trivial functor $\triv\colon \cM\rightarrow \cM^{gr, \epsilon}$.
\end{defn}
Explicitly, it is given by $A\mapsto (A\otimes_{k[\epsilon]} k)^{\G_m}$. We have a strict model of $|-|^l$ which is a functor $M^{gr, \epsilon}\rightarrow M$ given by $A\mapsto (A\otimes_{k[\epsilon]} \tilde{k})^{\G_m}$. In the case $M=\dg_k$ we have explicitly
\[|A|^l \cong \bigoplus_{n\leq 0} A(n)\]
with the differential given by twisting the original differential by $\epsilon$.

We have the following statements about our strict models of the realization functors.
\begin{prop}
The realization functor $|-|\colon M^{gr, \epsilon}\rightarrow M$ preserves weak equivalences between fibrant objects.
\label{prop:rightrealizationweq}
\end{prop}
\begin{proof}
The internal Hom $\iHom_{M^{gr, \epsilon}}(-, -)$ is a Quillen bifunctor in the projective model structure. Therefore, $\iHom_{M^{gr, \epsilon}}(\bu_M\otimes \tilde{k}, -)$ preserves weak equivalences between fibrant objects.
\end{proof}

\begin{prop}
The left realization functor $|-|^l\colon M^{gr,\epsilon}\rightarrow M$ preserves weak equivalences.
\label{prop:leftrealizationweq}
\end{prop}
\begin{proof}
By definition the functor on $A\in M^{gr, \epsilon}$ is given by
\[|A|^l=(A \otimes_{\bu_M\otimes k[\epsilon]} (\bu_M\otimes \tilde{k}))^{\G_m},\]

The functor of $\G_m$-invariants clearly preserves weak equivalences, so we just need to show that the functor
\[\mod_{k[\epsilon]}(M^{gr})\rightarrow M^{gr}\]
given by
\[A\mapsto A\otimes_{k[\epsilon]} \tilde{k}\]
preserves weak equivalences.

But $\tilde{k}$ is cofibrant as a $k[\epsilon]$-module and hence $\bu_M\otimes \tilde{k}$ is flat over $\bu_M\otimes k[\epsilon]$ by our assumptions on the model category.
\end{proof}

An important feature of the realization functors are the natural filtrations that they carry. For a graded mixed object $A\in M^{gr, \epsilon}$ we define $|A|^{\geq n}$ to be the realization of $A\otimes k(-n)$. Similarly, we define $|A|^{l, \leq n}$ to be the left realization of $A\otimes k(-n)$

\begin{prop}
Suppose $A\in M^{gr, \epsilon}$ is a graded mixed object. Then
\[|A|^{\geq (n+1)}\rightarrow |A|^{\geq n}\rightarrow A(n)\]
is a cofiber sequence.

Similarly,
\[A(n)\rightarrow |A|^{l, \leq n}\rightarrow |A|^{l, \leq(n-1)}\]
is a fiber sequence.
\label{prop:realizationfiber}
\end{prop}

Since $|-|^l$ is left adjoint to a symmetric monoidal functor, it is naturally an oplax symmetric monoidal functor.
\begin{prop}
Suppose $A_1, A_2\in M^{gr, \epsilon}$ are two objects concentrated in non-positive degrees. Then the natural morphism
\[|A_1\otimes A_2|^l\rightarrow |A_1|^l\otimes |A_2|^l\]
is an isomorphism.
\end{prop}

We denote by $M^{\leq 0}\subset M$ the full subcategory of objects concentrated in non-positive weights and similarly for $M^{\leq 0, \epsilon}\subset M^{gr, \epsilon}$. The previous Proposition implies that we have a well-defined functor
\[|-|^l\colon \CAlg(M^{\leq 0, \epsilon})\rightarrow \CAlg(M).\]

We also have a functor of Tate realization which combines both left and ordinary realizations.
\begin{defn}
The \emph{Tate realization functor}
\[|-|^t\colon \cM^{gr, \epsilon}\longrightarrow \cM\]
is defined to be
\[|A|^t = \underset{i\geq 0}{\colim} |A\otimes k(-i)|.\]
\end{defn}

Finally, we will need a certain weakening of the notion of a graded mixed object. By definition the data of a mixed structure on a graded object $A\in M^{gr}$ boils down to a morphism of graded Lie algebras
\[\epsilon_A\colon k(2)[-1]\longrightarrow \Hom_{M^{gr}}(A, A).\]

Replacing strict morphisms by $\infty$-morphisms we arrive at the following definition:
\begin{defn}
A \emph{weak graded mixed object} is a graded object $A\in M^{gr}$ equipped with endomorphisms $\epsilon_1,\epsilon_2,\dots$ of $A$ where $\epsilon_i$ has degree $1$ and weight $i$ such that
\[(\d + \epsilon_1 + \epsilon_2 + \dots)^2 = 0.\]
\end{defn}

Note that graded mixed objects correspond to the case where $\epsilon_i = 0$ for $i>1$. Similarly, we can define $\infty$-morphisms of weak graded mixed objects:
\begin{defn}
An \emph{$\infty$-morphism} of weak graded mixed objects $f\colon A\rightarrow B$ is given by a collection of maps $f_0,f_1,\dots\colon A\rightarrow B$, where $f_i$ has degree $0$ and weight $i$ such that for $f= f_0 + f_1 + \dots$ we have
\[f\circ (\d + \epsilon^A_1 + \dots) = (\d + \epsilon^B_1 + \dots)\circ f.\]
\end{defn}

It is obvious that weak graded mixed objects form a category whose localization is equivalent to $\cM^{gr, \epsilon}$, see \cite[Section 3.3.4]{CPTVV}.

\subsection{Operads}

Our conventions about operads follow those of \cite{DR} and \cite{LV}. All operads we consider are operads in chain complexes.

Recall that a \emph{symmetric sequence} $V$ is a sequence of chain complexes $V(n)\in\dg_k$ together with an action of $S_n$ on $V(n)$. The category of symmetric sequences is monoidal with respect to the composition product and an operad is an algebra in the category of symmetric sequences. We denote the category of operads as $\Op_k$.

We denote by $\Tree_m(n)$ the groupoid of planar trees with labeled $n$ incoming edges and $m$ vertices. The morphisms are not necessarily planar isomorphisms between trees. For instance, the groupoid $\Tree_2(n)$ has components parametrized by $(p,n-p)$-shuffles $\sigma$ for any $p$, where a shuffle $\sigma$ corresponds to the tree $\t_\sigma$ as shown in Figure \ref{fig:shuffletree}.

\begin{figure}
\begin{minipage}{.4\textwidth}
\centering
\begin{tikzpicture}
\node[w] (v1) at (0, 0) {};
\node (root) at (0, -0.5) {};
\node[w] (v2) at (1, 1) {};
\node (v11) at (-1, 2) {$2$};
\node (v12) at (-0.4, 2) {$4$};
\node (v21) at (0.4, 2) {$1$};
\node (v22) at (1, 2) {$3$};
\node (v23) at (1.6, 2) {$5$};
\draw (v1) edge (root);
\draw (v1) edge (v2);
\draw (v1) edge (v11);
\draw (v1) edge (v12);
\draw (v2) edge (v21);
\draw (v2) edge (v22);
\draw (v2) edge (v23);
\end{tikzpicture}
\caption{The tree $\t_\sigma$ corresponding to a $(2,3)$-shuffle $\sigma$.} 
\label{fig:shuffletree}
\end{minipage}
\qquad
\begin{minipage}{.4\textwidth}
\centering
\begin{tikzpicture}
\node[w] (v0) at (0, 0) {};
\node (root) at (0, -0.5) {};
\node[w] (v1) at (-0.5, 0.5) {};
\node[w] (v2) at (0, 0.5) {};
\node[w] (v3) at (0.5, 0.5) {};
\node (v11) at (-0.7, 1) {};
\node (v12) at (-0.5, 1) {};
\node (v13) at (-0.3, 1) {};
\node (v21) at (-0.1, 1) {};
\node (v22) at (0.1, 1) {};
\node (v31) at (0.4, 1) {};
\node (v32) at (0.6, 1) {};
\draw (v0) edge (root);
\draw (v0) edge (v1);
\draw (v0) edge (v2);
\draw (v0) edge (v3);
\draw (v1) edge (v11);
\draw (v1) edge (v12);
\draw (v1) edge (v13);
\draw (v2) edge (v21);
\draw (v2) edge (v22);
\draw (v3) edge (v31);
\draw (v3) edge (v32);
\end{tikzpicture}
\caption{A pitchfork in $\Isom_\pitchfork(7, 3)$.}
\label{fig:pitchfork}
\end{minipage}
\end{figure}

We will also be interested in the set $\Isom_\pitchfork(n, r)$ of \emph{pitchforks} with $n$ incoming edges and $r+1$ vertices, see Figure \ref{fig:pitchfork} for an example and \cite[Section 2]{DW2} for more details. The groupoid $\Tree_3(n)$ has trees of two kinds: pitchforks in $\Isom_\pitchfork(n, 2)$ and the complement $\Tree_3^0(n)$.

Given a tree $\t\in\Tree_m(n)$ and a symmetric sequence $\cO$ we define $\cO(\t)$ to be the tensor product
\[\cO(\t)=\bigotimes_i \cO(n_i)\]
where the tensor product is over the vertices of $\t$ and $n_i$ is the number of incoming edges at vertex $i$.

Given an operad $\cO$, a tree $\t\in\Tree_m(n)$ defines a multiplication map
\[m_\t\colon \cO(\t)\rightarrow \cO(n).\]
Similarly, for a cooperad $\cC$ we have a comultiplication map
\[\Delta_\t\colon \cC(n)\rightarrow \cC(\t).\]

All cooperads are assumed to be coaugmented and conilpotent and we denote by $\cC\cong \overline{\cC}\oplus \bu$ the natural splitting.

Given a symmetric sequence $\cP$, the free operad $\Free(\cP)$ has operations parametrized by trees $\t$ whose vertices are labeled by operations in $\cP$. Given a cooperad $\cC$ we define its cobar complex $\Omega\cC$ to be the free operad on $\overline{\cC}[-1]$. The differential on the generators $X\in\overline{\cC}(n)[-1]$ is given by
\begin{equation}
\d X = -\s\d_1 (\s^{-1}X) - \sum_{\t\in\pi_0(\Tree_2(n))} (\s\otimes\s)(\t, \Delta_{\t}(\s^{-1} X))
\label{eq:cobardifferential}
\end{equation}
where $\d_1$ is the differential on the symmetric sequence $\cC$.

The following lemma is standard, and its proof can be found for example in \cite[Proposition 6.5.6]{LV}.
\begin{lm}
The cobar differential $\d$ on $\Omega\cC$ squares to zero.
\label{lm:cobardifferential}
\end{lm}

We will also need a slight generalization of the above construction.
\begin{defn}
A \emph{curved cooperad} is a dg cooperad $\cC$ together with a morphism of symmetric sequences $\theta\colon \cC\rightarrow \bu[2]$ such that
\[(\theta\otimes \id_{\cC} - \id_{\cC}\otimes \theta) \Delta(x) = 0\]
for any $x\in\cC$.
\end{defn}
Note that this definition is slightly more restrictive than the corresponding notion in \cite{HM}, but it will suffice for our purposes. Given a curved cooperad $\cC$ we can also consider its cobar complex $\Omega\cC$, we refer to \cite[Section 3.3.6]{HM} for explicit formulas for the differentials on $\Omega\cC$.

Given an operad $\cO$ and a complex $A$, we define the free $\cO$-algebra on $A$ to be
\[\cO(A) = \bigoplus_n (\cO(n)\otimes A^{\otimes n})_{S_n}.\]

Similarly, for a cooperad $\cC$ and a complex $A$, we define the cofree conilpotent $\cC$-coalgebra on $A$ to be
\[\cC(A) = \bigoplus_n (\cC(n)\otimes A^{\otimes n})^{S_n}.\]

We will also be interested in colored symmetric sequences and colored operads. Let $\cV$ be a set. A $\cV$-colored symmetric sequence is a collection of complexes $\cV(v_1^{\otimes n_1}\otimes \cdots \otimes v_m^{\otimes n_m}, v_0)$ for every collection of elements $v_0,v_1,\dots, v_m\in \cV$ together with an action of $S_{n_1}\times \dots\times S_{n_m}$. As before, the category of $\cV$-colored symmetric sequences has a composition product and a $\cV$-colored operad is defined to be an algebra object in the category of $\cV$-colored symmetric sequences. We denote the category of colored operads by $\cV\Op_k$; in particular, if the set of colors has two elements, we denote it by $2\Op_k$.

The following theorem is due to Hinich \cite{Hi} in the uncolored case; the colored case is treated in \cite{BM}, \cite{Cav} and \cite{PS}.

\begin{thm}
The category of (colored) dg operads $\cV\Op_k$ has a model structure which is transferred from the model structure on $\cV$-colored symmetric sequences by the free-forgetful adjunction.
\end{thm}

Considering the coradical filtration on a conilpotent cooperad $\cC$, one has the following:
\begin{prop}
Let $\cC$ be a conilpotent (colored) cooperad. Then $\Omega\cC$ is cofibrant.
\end{prop}

Here are our main examples of operads and cooperads:
\begin{itemize}
\item If $A$ is an object of $M$, $\End_A$ is a dg operad with $\End_A(n) = \Hom_M(A^{\otimes n}, A)$. Similarly, if we have a pair of objects $A,B\in M$, then $\End_{A,B}$ is a $\{\cA, \cB\}$-colored dg operad with
\[\End_{A, B}(\cA^{\otimes n}\otimes \cB^{\otimes m}, \cA) = \Hom(A^{\otimes n}\otimes B^{\otimes m}, A)\] and similarly for $\End_{A, B}(-, \cB)$.

\item $\Comm$ is the operad of unital commutative algebras, $\Lie$ is the operad of Lie algebras. $\bP_n$ is the operad of unital shifted Poisson algebras with the commutative multiplication of degree 0 and the Poisson bracket of degree $1-n$. We denote by $\Comm^{\nu}$ and $\bP_n^{\nu}$ the non-unital versions of the operads $\Comm$ and $\bP_n$.

\item The operads $\Lie$ and $\bP_n$ can also be upgraded to operads in graded complexes where we set the weight of the bracket to be $-1$ and the weight of the multiplication to be zero.

\item $\coComm$ is the cooperad of non-counital cocommutative coalgebras, $\coLie$ is the cooperad of Lie coalgebras. $\coP_n$ is the cooperad of non-counital shifted Poisson coalgebras with the cocommutative comultiplication of degree 0 and the Poisson cobracket of degree $1-n$. We denote by $\coComm^{\cu}$ and $\coP_n^{\cu}$ the counital versions of the cooperads $\coComm$ and $\coP_n$.

\item If $\cO$ is a symmetric sequence, we denote by $\cO\{n\}$ the symmetric sequence defined by
\[\cO\{n\}(m) = \cO(m)\otimes \sgn_m^{\otimes n}[n(m-1)],\]
where $\sgn_m$ is the sign representation of $S_m$. If $\cO$ is a (co)operad, then so is $\cO\{n\}$. For instance, $\Lie\{n\}$ is the operad of Lie algebras with bracket of degree $-n$.
\end{itemize}

Given a dg operad $\cO$, we can consider $\cO$-algebras in $M$. We denote the corresponding category by $\alg_{\cO}(M)$ and its localization by $\balg_{\cO}(\cM)$. We also introduce shorthands
\[\alg_{\Lie}^{gr} = \alg_{\Lie}(\dg_k^{gr}),\qquad \alg(M)=\alg_{\Ass}(M).\]

\subsection{Lie algebras}

Given a nilpotent $L_\infty$ algebra $\g$, the set of Maurer--Cartan elements is defined to be the set of degree $1$ elements $x\in\g$ such that
\[\d x + \sum_{n\geq 2} \frac{1}{n!}[x, \dots, x]_n = 0.\]
When there is no confusion, we will omit the subscript on the bracket. Similarly, if $\g$ is a curved nilpotent $L_\infty$ algebra with curvature $\theta$, the Maurer--Cartan equation is
\[\theta + \d x + \sum_{n\geq 2} \frac{1}{n!}[x, \dots, x]_n = 0.\]

Let $\Omega_\bullet$ be the cosimplicial commutative algebra of polynomial differential forms on simplices. For instance, $\Omega_0 = k$ and $\Omega_1 = k[x, y]$ with $\deg(x) = 0$, $\deg(y) = 1$ and $\d x = y$.
\begin{defn}
The \emph{space of Maurer--Cartan elements} $\underline{\MC}(\g)$ is the simplicial set of Maurer--Cartan elements in $\g\otimes \Omega_\bullet$.
\end{defn}

Now suppose that $\g$ is a (curved) $L_\infty$ algebra equipped with a decreasing filtration $\g = \g^a\supset \g^{a+1}\supset \dots$ such that $[\g^{a_1}, \dots, \g^{a_n}]_n\subset \g^{\sum a_i + 1 - n}$ and such that the quotients $\g / \g^a$ are nilpotent. We call such filtrations \emph{admissible}. For an $L_\infty$ algebra $\g$ equipped with such an admissible filtration, we define
\[\underline{\MC}(\g) = \lim_a \underline{\MC}(\g / \g^a).\]

\begin{defn}
Let $\g$ be a (curved) $L_\infty$ algebra and $x\in\g$ a Maurer--Cartan element. The \emph{$L_\infty$ algebra $\g$ twisted by $x$} has the same underlying graded vector space; the brackets are defined by
\[[x_1, \dots, x_n]_n = \sum_{k \geq 0} \frac{1}{k!} [x, \dots, x, x_1, \dots, x_n]_{n+k}.\]
\label{def:MCtwist}
\end{defn}

\begin{lm}
Let $\g_1$ and $\g_2$ be two admissible filtered $L_\infty$ algebras with a pair of morphisms $p\colon \g_1\rightarrow \g_2$ and $i\colon \g_2\rightarrow \g_1$ such that $p\circ i=\id_{\g_2}$. Then the homotopy fiber of
\[\underline{\MC}(\g_1)\rightarrow \underline{\MC}(\g_2)\]
at a Maurer--Cartan element $x\in\g_2$ is equivalent to the space of Maurer--Cartan elements in the $L_\infty$ algebra $\ker p$ twisted by the element $i(x)$.
\label{lm:MCfiber}
\end{lm}
\begin{proof}
By assumption the morphism $p\colon \g_1\rightarrow \g_2$ is surjective. By \cite[Theorem 3.2]{Ya1} the induced morphism $\underline{\MC}(\g_1)\rightarrow\underline{\MC}(\g_2)$ is a fibration of simplicial sets. Hence the homotopy fiber is equivalent to the strict fiber, i.e. the inverse limit of fibers of $\underline{\MC}(\g_1/\g_1^n)\rightarrow \underline{\MC}(\g_2/\g_2^n)$.

The set of $m$-simplices of the strict fiber of $\underline{\MC}(\g_1/\g_1^n)\rightarrow \underline{\MC}(\g_2/\g_2^n)$ at $x\in\g_2$ is isomorphic to the set of Maurer-Cartan elements $y\in \g_1\otimes \Omega_m$ such that $p(y) = x$. Using $i$ we can identify
\[\g_1\cong \g_2\oplus \ker p\]
as filtered $L_\infty$ algebras.

Therefore, the set of $m$-simplices of the strict fiber consists of elements $y_0\in\ker p$ satisfying the equation
\[\theta + \d(y_0 + i(x)) + \sum_{n\geq 2} \frac{1}{n!}[y_0 + i(x), \dots, y_0 + i(x)]_n.\]

Let us now expand this equation. The term not involving $y_0$ is
\[\theta + \d i(x) + \sum_{n\geq 2} \frac{1}{n!}[i(x), \dots, i(x)]_n = 0\]
by the Maurer--Cartan equation for $i(x)$. Therefore, we get the equation
\[\d y_0 + \sum_{n \geq 2}\frac{1}{(n-1)!} [i(x), \dots, i(x), y_0]_n + \sum_{n\geq 2}\frac{1}{n!}\sum_{k\geq 0}\frac{n!}{k!(n-k)!} [i(x), \dots, i(x), y_0, \dots, y_0]_n,\]
i.e. the Maurer--Cartan equation in the $L_\infty$ algebra $\ker p$ twisted by $i(x)$.
\end{proof}

Recall that the operad $\Lie$ is an operad in graded complexes where the weight of the bracket is $-1$. The operad of non-curved $L_\infty$-algebras can also be enhanced to an operad in graded complexes by assigning weight $1-n$ to the bracket $[-, \dots, -]_n$. Alternatively, one can consider the $\coComm\{1\}$ as a graded cooperad with coproduct of weight $1$ and define the $L_\infty$ operad as $\Omega(\coComm\{1\})$.

Given a graded $L_\infty$ algebra $\g$ we denote by
\[\g^{\geq m} = \prod_{n\geq m} \g(n)\]
its completion in weights $\geq m$. Then $\g^{\geq 2}$ carries an admissible filtration
\[\g^{\geq 2}\supset \g^{\geq 3} \supset \dots\]

A version of the following statement was proved by the first author in \cite[Section 4]{Me}. Let $k(2)[-1]$ be the trivial graded $L_\infty$ algebra in degree $1$ and weight 2.

\begin{prop}
Let $\g$ be a graded $L_\infty$ algebra in $\dg_k$. There is an equivalence of spaces
\[\Map_{\alg_{\mathrm{L}_\infty}^{gr}}(k(2)[-1], \g)\cong \underline{\MC}(\g^{\geq 2}).\]

Similarly, if $\g$ is a graded dg Lie algebra, then there is an equivalence of spaces
\[\Map_{\alg_{\Lie}^{gr}}(k(2)[-1], \g)\cong \underline{\MC}(\g^{\geq 2}).\]
\label{prop:MCgradeddglie}
\end{prop}
\begin{proof}
Recall that in the model category of $L_\infty$ algebras every object is fibrant, so we just need to find a cofibrant replacement $L$ for $k(2)[-1]$.

By \cite[Lemma 6.5.14]{LV} the symmetric sequence $L_\infty\circ_{\alpha} \coComm\{1\}$ equipped with the Koszul differential $\alpha$ is quasi-isomorphic to the unit symmetric sequence, so the natural morphism
\[L_\infty(\overline{\Sym}_\bullet(V[1])[-1])\rightarrow V\]
is a quasi-isomorphism for any complex $V$. Here $\overline{\Sym}_\bullet$ is the reduced symmetric algebra and the free $L_\infty$-algebra $L_\infty(\overline{\Sym}_\bullet(V[1])[-1])$ is equipped with the Koszul differential $\alpha$.

Let $V = k(2)[-1]$, then $\overline{\Sym}_\bullet(V[1])[-1] = \mathrm{span}\{p_2, p_3, \dots\}$, where $p_n$ has weight $n$ and degree $1$. Let $p=\sum_{i=2}^\infty p_i$. Then the Koszul differential gives
\[\d p + \sum_{n=2}^\infty [p, \dots, p] = 0.\]

Therefore, if we define $L_0$ to be the free graded $L_\infty$ algebra on the generators $p_2, p_3 , \dots$ equipped with the differential as above, then the natural morphism $L_0\rightarrow k(2)[-1]$ given by projecting on $p_2$ is a quasi-isomorphism. Moreover, by construction $L_0$ is cofibrant.

Therefore, one has equivalences of spaces
\begin{align*}
\Map_{\alg_{\mathrm{L}_\infty}^{gr}}(k(2)[-1], \g) &\cong \Hom_{\alg_{\mathrm{L}_\infty}^{gr}, \bullet}(L_0, \g) \\
&\cong \Hom_{\alg_{\mathrm{L}_\infty}^{gr}}(L_0, \g\otimes \Omega_\bullet).
\end{align*}

The latter Hom is easy to compute and it exactly gives the set of Maurer--Cartan elements in $\g^{\geq 2}\otimes \Omega_\bullet$.
\end{proof}

\section{Operadic resolutions}
\label{sect:operadicres}

In this section we collect some useful results which describe spaces of $\cO$-algebra structures for an operad $\cO$ and spaces of $\cO$-algebra morphisms.

\subsection{Deformations of algebras}
\label{sect:deformationalgebras}

Let $\cC$ be a coaugmented dg cooperad so that we can split $\cC\cong \overline{\cC}\oplus \bu$. For an operad $\cP$ we introduce the convolution algebra $\Conv(\cC, \cP)$, a dg Lie algebra, as follows. As a complex it is defined to be
\[\Conv(\cC, \cP)=\prod_n \Hom_{S_n}(\overline{\cC}(n), \cP(n)).\]

For brevity we denote
\[\Conv(\cC; A) = \Conv(\cC, \End_A).\]

We introduce a pre-Lie structure on $\Conv(\cC, \cP)$ by
\[(f\bullet g)(X) = \sum_{\t\in\pi_0(\Tree_2(n))} \mu_\t((f\otimes g)\Delta_\t(X))\]
for any $f,g\in\Conv(\cC, \cP)$ and $X\in\cC(n)$. The Lie bracket is defined to be
\[[f, g] = f\bullet g - (-1)^{|f||g|}g\bullet f.\]

Here is a more explicit description of the pre-Lie structure. Recall that operations of arity $m$ in the pre-Lie operad are parametrized by rooted trees with $m$ vertices \cite{CL}. Given a rooted tree $\t$ we denote by $\Tree_m(\t, n)$ the groupoid of trees obtained by attaching incoming edges at each vertex of the tree $\t$ such that the total number of incoming edges is $n$ (Fig. \ref{fig:prelie}). The action of the rooted tree $\t$ on the elements $f_1, \dots, f_m\in \Conv(\cC; A)$ is given by the sum over the trees $\t'\in\pi_0(\Tree_m(\t, n))$ where each term in the sum is given by the composition associated to the tree $\t'$ with vertices labeled by the elements $f_i$.

\begin{figure}
\begin{tikzpicture}
\node[w] (v1) at (0, 0) {$1$};
\node (root) at (0, -1) {};
\node[w] (v2) at (-1, 1) {$2$};
\node[w] (v3) at (1, 1) {$3$};
\draw (v1) edge (v2);
\draw (v1) edge (v3);
\draw (v1) edge (root);
\end{tikzpicture}
\qquad
\begin{tikzpicture}
\node[w] (v1) at (0, 0) {$1$};
\node (root) at (0, -1) {};
\node[w] (v2) at (-1, 1) {$2$};
\node[w] (v3) at (1, 1) {$3$};
\node (v11) at (-0.2, 1) {};
\node (v12) at (0, 1) {};
\node (v13) at (0.2, 1) {};
\node (v21) at (-1.2, 2) {};
\node (v22) at (-0.8, 2) {};
\node (v31) at (1.2, 2) {};
\node (v32) at (0.8, 2) {};
\draw (v1) edge (root);
\draw (v1) edge (v2);
\draw (v1) edge (v3);
\draw (v1) edge (v11);
\draw (v1) edge (v12);
\draw (v1) edge (v13);
\draw (v2) edge (v21);
\draw (v2) edge (v22);
\draw (v3) edge (v31);
\draw (v3) edge (v32);
\end{tikzpicture}
\caption{A rooted tree $\t$ and an element of $\Tree_3(\t, 7)$.}
\label{fig:prelie}
\end{figure}
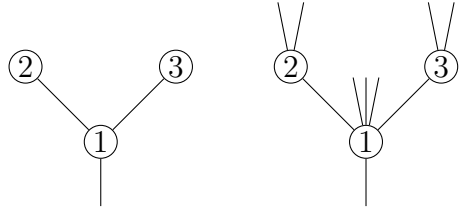

For example, consider the tree $\t$ defining the pre-Lie bracket. Then $\Tree_2(\t, n)\cong \Tree_2(n)$ whose connected components are parametrized by trees $\t_\sigma$ associated to the shuffles $\sigma\in S_{p, n-p}$. Therefore,
\begin{equation}
(f\bullet g)(X; a_1, \dots, a_n) = \sum_{p=0}^{n}\sum_{\sigma\in S_{p,n-p}} \pm f(X^{\t_\sigma}_{(1)}; g(X^{\t_\sigma}_{(2)}; a_{\sigma(1)}, \dots, a_{\sigma(p)}), a_{\sigma(p+1)}, \dots, a_{\sigma(n)}),
\label{eq:convprelie}
\end{equation}
where the sign arises from the permutation of $\{f,g,X^{\t_\sigma}_{(1)},X^{\t_\sigma}_{(2)},a_1,\dots,a_n\}$ and the tree $\t_\sigma$ is the tree corresponding to the shuffle $\sigma$.

If $\cP$ is an operad in $M$, we can enhance $\Conv(\cC, \cP)$ to a Lie algebra in $M$ by considering the internal Hom in $M$ and similarly for $\Conv(\cC; A)$ in the case $A$ is an object of $M$.

If $\cC$ is a graded cooperad, we introduce a graded Lie algebra structure on $\Conv(\cC, \cP)$ by considering the internal grading on $\cC$ and putting $\cP$ in weight 1. In this way $\Conv(\cC, \cP)$ acquires a Lie structure of weight $-1$. Note that graded morphisms $\Omega\cC\rightarrow \cP$ give rise to elements of $\Conv(\cC, \cP)$ which are pure of weight $1$.

Finally, if $\cC$ is a curved cooperad, we obtain a curved Lie algebra structure on $\Conv(\cC, \cP)$ as follows. The curving on $\Conv(\cC, \cP)$ is the weight $1$, degree $2$ element obtained as a composite
\[\cC(1)\stackrel{\theta}\rightarrow k[2]\rightarrow \cP(1)[2],\]
where the second morphism is the unit morphism in $\cP$.

The following statement is proved by considering the coradical filtration on $\cC$:
\begin{lm}
Suppose $\cC$ is a conilpotent (curved) cooperad. Then $\Conv(\cC, \cP)$ is pronilpotent.
\label{lm:conilpotentadmissible}
\end{lm}

Using this Lemma we have the following statement.

\begin{prop}
Assume $\cO$ is an operad with a weak equivalence $\Omega\cC\stackrel{\sim}\rightarrow \cO$, where $\cC$ is a conilpotent (curved) cooperad. Then we have an equivalence of spaces
\begin{align*}
\Map_{\Op_k}(\cO, \cP)&\cong \underline{\MC}(\Conv(\cC, \cP)).
\end{align*}
\label{prop:convolutionla}
\end{prop}
\begin{proof}
Note that since $\cC$ is conilpotent, $\Omega\cC$ is cofibrant. We have a sequence of equivalences of spaces
\[\Map_{\Op_k}(\cO, \cP)\cong \Map_{\Op_k}(\Omega \cC, \cP)\cong \iHom_{\Op_k,\bullet}(\Omega\cC, \cP)\cong \Hom_{\Op_k}(\Omega\cC, \cP\otimes \Omega_\bullet).\]

An operad morphism $f\colon \Omega\cC\rightarrow \cP$ is uniquely specified by a degree 0 map of symmetric sequences $f_0\colon \overline{\cC}[-1]\rightarrow \cP$ satisfying the equation
\begin{align*}
\d (f_0(\s X)) &= f(\d \s X)\\
&= f\left(-\s\d X - \sum_{\t\in\pi_0(\Tree_2(n))}(\s\otimes \s)(\t, \Delta_\t(X))\right) \\
&= -f_0(\s\d X) - f\left(\sum_{\t\in\pi_0(\Tree_2(n))}(\s\otimes \s)(\t, \Delta_\t(X))\right)
\end{align*}
for any $\s X\in \cC_\circ(n)[-1]$. Since $f$ is a morphism of operads, the last term can also be written in terms of $f_0$, so we obtain
\[\d (f_0(\s X)) + f_0(\s\d X) + \sum_{\t\in\pi_0(\Tree_2(n))} \mu_\t((f_0\s\otimes f_0\s)\Delta_\t(X)) = 0.\]

Finally, identifying degree 0 maps $f_0\colon \overline{\cC}[-1]\rightarrow \cP$ with degree 1 maps $f_0\s\colon \overline{\cC}\rightarrow \cP$ we get exactly the Maurer--Cartan equation in $\Conv(\cC, \cP)$. Since the simplicial set of Maurer--Cartan elements in a dg Lie algebra $\g$ is defined to be the set of Maurer--Cartan elements in $\g\otimes\Omega_\bullet$ and $\Conv(\cC, \cP\otimes \Omega_\bullet)\cong \Conv(\cC, \cP)\otimes \Omega_\bullet$, we obtain an equivalence of spaces
\[\Map_{\Op_k}(\cO, \cP)\cong \underline{\MC}(\Conv(\cC, \cP)).\]
\end{proof}

We will also need a variant of the convolution algebra
\[\Conv^0(\cC, \cP)=\prod_n \Hom_{S_n}(\cC(n), \cP(n))\]
which does not use a coaugmentation on $\cC$.

The Lie bracket on $\Conv(\cC, \cP)$ extends to one on $\Conv^0(\cC, \cP)$. Note that in contrast to $\Conv(\cC, \cP)$, the Lie algebra $\Conv^0(\cC, \cP)$ is not pronilpotent.

\subsection{Harrison complex}

Suppose $A$ is a commutative algebra in $M$. We will begin by constructing a canonical resolution of $A$, i.e. a graded mixed commutative algebra $A^\epsilon$, free as a graded commutative algebra, together with a weak equivalence $|A^\epsilon|^l\stackrel{\sim}\rightarrow A$. For this we will use the canonical cobar-bar resolution, the reader is referred to \cite[Section 11.2]{LV} for details.

Let $\coLie^\theta$ be the cooperad of curved Lie coalgebras. Then as in \cite[Section 6.1]{HM} one can construct a weak equivalence of operads
\[\Omega(\coLie^\theta\{1\})\rightarrow \Comm.\]

The cooperad $\coLie^\theta$ admits a weight grading where we set the weights of comultiplication and curving to be $-1$. Under this grading $\coLie^\theta$ is concentrated in non-positive weights. Similarly, $\Comm$ has a filtration by the number of generating operations. Moreover, the morphism $\Omega(\coLie^\theta\{1\})\rightarrow \Comm$ is compatible with the filtrations.

We define $A^\epsilon = \Sym(\coLie^\theta\{1\}(A)) = \Comm\circ \coLie^\theta\{1\}\circ A$ as a commutative algebra, where $\circ$ is the composition product of symmetric sequences with $A$ considered as a symmetric sequence in arity 0. We define the grading on $A^\epsilon$ using the grading on $\coLie^\theta\{1\}$. The mixed structure on $A^\epsilon$ consists of two terms and coincides with the differential on the cobar-bar resolution for which we refer to \cite[Section 5.2]{HM}.

The counit morphism $\coLie^\theta\{1\}\rightarrow \bu$ induces a morphism
\[\coLie^\theta\{1\}(A)\rightarrow A\]
which defines a morphism of graded commutative algebras
\[\Sym(\coLie^\theta\{1\}(A))\rightarrow A.\]
It is easy to see that it is compatible with the mixed structure, hence one obtains a morphism
\[f\colon |A^\epsilon|^l\rightarrow A.\]

The following statement is \cite[Theorem 11.3.6]{LV} and \cite[Theorem 15]{LG}:
\begin{prop}
Let $A$ be a commutative algebra in $M$ cofibrant as an object of $M$. Then the morphism $f\colon |A^\epsilon|^l\rightarrow A$ is a weak equivalence.
\label{prop:cobarbarwe}
\end{prop}

\begin{lm}
Let $A$ be a commutative algebra in $M$ cofibrant as an object of $M$. Then $|A^\epsilon|^l$ is a cofibrant commutative algebra.
\label{lm:cobarbarcofibrant}
\end{lm}
\begin{proof}
Consider a filtration on the resolution $A^\epsilon$ as follows. Let $A^\epsilon_n$ be the symmetric algebra on $\coLie^\theta\{1\}(A)$ in weights at least $-n$; the mixed structure on $A^\epsilon$ restricts to one on $A^\epsilon_n$. In particular, $A^\epsilon_0 = \Sym(A)$ with the trivial mixed structure.

Since $A$ is cofibrant as an object of $M$, the object
\[\bigoplus_{0\leq m\leq n} (\coLie^\theta\{1\}(m)\otimes A^{\otimes m})^{S_m}\]
is cofibrant as well. $|A^\epsilon|^l$ is given as the colimit of the direct system
\[|A^\epsilon_0|^l\rightarrow |A^\epsilon_1|^l\rightarrow \dots,\]
but each arrow is a cofibration of commutative algebras and hence $|A^\epsilon|^l\in\CAlg(M)$ is cofibrant.
\end{proof}

\begin{lm}
Let $B\in \CAlg(M^{\leq 0, \epsilon})$ be a graded mixed commutative algebra in $M$ concentrated in non-positive weights. We have an isomorphism
\[|\Omega^1_B|^l\cong \Omega^1_{|B|^l}\]
of $|B|^l$-modules, where on the left we use the functor
\[|-|^l\colon \mod_B(M^{gr, \epsilon})\rightarrow \mod_{|B|^l}(M).\]
\label{lm:kahlerrealization}
\end{lm}
\begin{proof}
Recall the following explicit construction of the module of K\"{a}hler differentials. Let $A$ be a commutative algebra in $M$.

Denote by $i\colon \Sym_2(A)\rightarrow A\otimes A$ the space of $S_2$-invariants. We denote $m_s\colon \Sym_2(A)\rightarrow A$ the multiplication map given by $m_s = \frac{1}{2}m\circ i$. The module of K\"{a}hler differentials $\Omega^1_A\in \mod_A(M)$ is defined to be the coequalizer of $A$-modules
\[\Omega^1_A=\coeq\left(
\xymatrixcolsep{5pc}\xymatrix{
A\otimes \Sym_2(A) \ar@<.5ex>^-{(m\otimes \id)\circ(\id\otimes i)}[r] \ar@<-.5ex>_-{\id\otimes m_s}[r] & A\otimes A
}
\right),
\]
where $A\otimes \Sym_2(A)$ and $A\otimes A$ are the free $A$-modules on $\Sym_2(A)$ and $A$ respectively.

Since $|-|^l$ is a left adjoint, it preserves colimits and therefore the claim follows from the explicit description of the module of K\"{a}hler differentials given above.
\end{proof}

We are now going to define the Harrison chain and cochain complexes for a commutative algebra $A\in\CAlg(M)$ which will be certain graded mixed $A$-modules whose realizations represent the cotangent and tangent complex respectively.

Define
\[\Harr_\bullet(A, A) = A\otimes \coLie\{1\}(A)\cong \Omega^1_{A^\epsilon}\otimes_{A^\epsilon} A \in \mod_A(M^{gr, \epsilon})\]
with the mixed structure coming from the one on $\Omega^1_{A^\epsilon}$. By construction we have a morphism of graded mixed $A$-modules
\[\Harr_\bullet(A, A)\rightarrow \Omega^1_A\]
given by
\[f\otimes g\mapsto f\ddr g\]
in weight $0$, where we consider the trivial graded mixed structure on $\Omega^1_A$.

\begin{prop}
Suppose $A\in\CAlg(M)$ is a cofibrant commutative algebra in $M$. Then the morphism
\[\Harr_\bullet(A, A)\rightarrow \Omega^1_A\]
induces a weak equivalence
\[|\Harr_\bullet(A, A)|^l\rightarrow \Omega^1_A\]
of $A$-modules.
\label{prop:harrisoncotangent}
\end{prop}
\begin{proof}
By \cite[Proposition A.1.4]{CPTVV} the forgetful functor $\CAlg(M)\rightarrow M$ preserves cofibrant objects, so $A$ is cofibrant as an object of $M$.

By Proposition \ref{prop:cobarbarwe} and Lemma \ref{lm:cobarbarcofibrant} the morphism of $A$-modules
\[\Omega^1_{|A^\epsilon|^l}\otimes_{|A^\epsilon|^l} A\rightarrow \Omega^1_A\]
is a weak equivalence.

By Lemma \ref{lm:kahlerrealization} we get an isomorphism
\[\Omega^1_{|A^\epsilon|^l}\otimes_{|A^\epsilon|^l} A\cong |\Omega^1_{A^\epsilon}|^l\otimes_{|A^\epsilon|^l} A.\]

But since $|-|^l$ preserves colimits and is monoidal,
\[|\Omega^1_{A^\epsilon}|^l\otimes_{|A^\epsilon|^l} A\cong |\Omega^1_{A^\epsilon}\otimes_{A^\epsilon} A|^l=\Harr_\bullet(A, A).\]
\end{proof}

Let us now introduce the Harrison cochain complex. Let
\[\Harr^\bullet(A, A) = \iHom_{\mod_A(M^{gr, \epsilon})}(\Harr_\bullet(A, A), A)\cong \iHom_M(\coLie^\theta\{1\}(A), A).\]

As graded objects we can identify
\[\Harr^\bullet(A, A)\cong \Conv(\coLie^\theta\{1\}; A)(-1).\]

The multiplication and unit on $A$ defines an element $m\in\Conv(\coLie^\theta\{1\}; A)$ satisfying the curved Maurer--Cartan equation. It is not difficult to check that the mixed structure on $\Harr^\bullet(A, A)$ is given by $[m, -]$.

Let $\Der^{int}(A, A)\in \alg_{\Lie}(M)$ be the Lie algebra of derivations of $A$. Consider the morphism of graded objects
\[\Der^{int}(A, A)\rightarrow \Harr^\bullet(A, A)\]
induced from the morphism $\Der^{int}(A, A)\rightarrow \iHom_M(A, A)$.

\begin{prop}
Suppose $A\in\CAlg(M)$ is a fibrant and cofibrant commutative algebra in $M$. Then the morphism
\[\Der^{int}(A, A)\rightarrow \Harr^\bullet(A, A)\]
induces a weak equivalence of Lie algebras in $M$
\[\Der^{int}(A, A)\rightarrow |\Harr^\bullet(A, A)|.\]
\label{prop:harrisontangent}
\end{prop}
\begin{proof}
The compatibility with the Lie brackets is obvious since the weight $0$ part of $\Harr^\bullet(A, A)$ is $\iHom_M(A, A)$ and the pre-Lie structure \eqref{eq:convprelie} restricts to the composition of endomorphisms.

The morphism $\Harr_\bullet(A, A)\rightarrow \Omega^1_A$ induces the morphism
\[\Der^{int}(A, A)\rightarrow\Harr^\bullet(A, A)\]
after applying $\iHom_{\mod_A(M)}(-, A)$. Since $A$ is fibrant, the functor $\iHom_{\mod_A(M)}(-, A)$ preserves weak equivalences between cofibrant objects. By construction $|\Harr_\bullet(A, A)|^l$ is a semi-free $A$-module, hence it is cofibrant. Since $A$ is cofibrant, $\Omega^1_A$ is also a cofibrant $A$-module. Therefore,
\[\Der^{int}(A, A)\rightarrow \iHom_{\mod_A(M)}(|\Harr_\bullet(A, A)|^l, A)\]
is a weak equivalence by Proposition \ref{prop:harrisoncotangent}. The statement follows by observing that the natural morphism
\[\iHom_{\mod_A(M)}(|\Harr_\bullet(A, A)|^l, A)\rightarrow |\iHom_{\mod_A(M^{gr, \epsilon})}(\Harr_\bullet(A, A), A)| = |\Harr^\bullet(A, A)|\]
is an isomorphism.
\end{proof}

\section{Brace construction}
\label{sect:braceconstruction}

\subsection{Braces}
\label{sect:braces}

Let $\cC$ be a cooperad. In Section \ref{sect:deformationalgebras} we have shown how to make the convolution algebras $\Conv(\cC; A)$ and $\Conv^0(\cC; A)$ into pre-Lie algebras for any complex $A$. In this section we introduce two important generalizations of this construction.

Recall from \cite[Definition 1.12]{Sa2} the following notion. Given a symmetric sequence $\cC$ we define $\cC^{\cu}$ to coincide with $\cC$ in arities at least 1 and $\cC^{\cu}(0)=\cC(0)\oplus k$ in arity 0.

\begin{defn}
A \emph{Hopf counital structure} on a (curved) cooperad $\cC$ is the structure of a (curved) Hopf cooperad on $\cC^{\cu}$ such that the natural projection $\cC^{\cu}\rightarrow \cC$ is a morphism of (curved) cooperads and such that the unit of the Hopf structure on $\cC^{\cu}(0)=\cC(0)\oplus k$ is given by inclusion into the second factor.
\end{defn}

\begin{remark}
If $\cC$ is a coaugmented cooperad, $\cC^{\cu}$ will not in general inherit the coaugmentation.
\end{remark}

Calaque and Willwacher extend the action of the pre-Lie operad $\preLie$ on $\Conv^0(\cC^{\cu}; A)$ to an action of the operad $\preLie_\cC$ whose definition we will now sketch. We refer to \cite[Section 3.1]{CW} for details.

Recall that the operations in the pre-Lie operad $\preLie$ are parametrized by rooted trees with numbered vertices. The operadic composition $\t_1\circ_m \t_2$ is given by grafting the root of the tree $\t_2$ into the $m$-th vertex of $\t_1$. Now suppose $\cC$ is a cooperad with a Hopf counital structure. The operad $\preLie_\cC$ has operations parametrized by rooted trees where each vertex is labeled by an operation of $\cC^{\cu}$ whose arity is equal to the number of incoming edges at the given vertex. The operadic composition is again given by grafting trees and it uses the Hopf cooperad structure on $\cC^{\cu}$. The action of $\preLie_{\cC}$ on $\Conv^0(\cC^{\cu}; A)$ is essentially the same as the pre-Lie structure on $\Conv^0(\cC^{\cu}; A)$ and uses the Hopf structure on $\cC^{\cu}$ to multiply the label on the rooted tree by the label in the convolution algebra.

Given a Maurer--Cartan element $f$ in a dg Lie algebra, the same Lie algebra with the differential $\d + [f, -]$ is still a dg Lie algebra. This is no longer true for pre-Lie algebras and one has to twist the operad. Using the general notion of twisting introduced in \cite{DW1} one can construct the brace operad $\Br_\cC$ which has operations parametrized by rooted trees where ``external'' vertices are labeled by elements of $\cC^{\cu}$ and the rest of the vertices, ``internal'' vertices, are labeled by elements of $\overline{\cC}[-1]$. In the pictures we will draw, external vertices are colored white and internal vertices are colored black. The generating operations of $\Br_\cC$ are shown in Figure \ref{fig:brgenerating}, where the root is labeled by an element of $\cC$. We refer to \cite[Section 9]{DW1} for an explicit description of the differential on $\Br_{\cC}$, but roughly it is obtained by the sum over all vertices of the following terms:
\begin{itemize}
\item If the vertex is external, we replace it by the first expression shown in Figure \ref{fig:brdifferential} and apply the composition in the pre-Lie operad.

\item If the vertex is internal, we replace it by the second expression shown in Figure \ref{fig:brdifferential} and apply the composition in the pre-Lie operad.

\item We discard all trees which have internal vertices with fewer than 2 children.
\end{itemize}

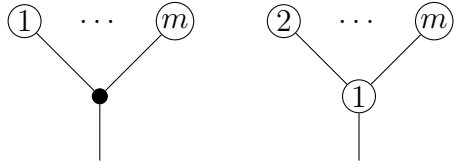
\begin{figure}
\begin{tikzpicture}
\node[b] (v0) at (0, 0) {};
\node (v) at (0, -1) {};
\node[w] (v1) at (-1, 1) {$1$};
\node at (0, 1) {\dots};
\node[w] (v2) at (1, 1) {$m$};
\draw (v0) edge (v);
\draw (v0) edge (v1);
\draw (v0) edge (v2);
\end{tikzpicture}
\qquad
\begin{tikzpicture}
\node[w] (v0) at (0, 0) {$1$};
\node (v) at (0, -1) {};
\node[w] (v1) at (-1, 1) {$2$};
\node at (0, 1) {\dots};
\node[w] (v2) at (1, 1) {$m$};
\draw (v0) edge (v);
\draw (v0) edge (v1);
\draw (v0) edge (v2);
\end{tikzpicture}
\caption{Generating operations of $\Br_\cC$.}
\label{fig:brgenerating}
\end{figure}

\begin{figure}
\begin{tikzpicture}
\node[b] (v0l) at (-0.5, 0) {};
\node (vl) at (-0.5, -1) {};
\node[w] (v1l) at (-0.5, 1) {$1$};
\node at (0, 0) {$-$};
\node[w] (v0r) at (0.5, 0) {$1$};
\node (vr) at (0.5, -1) {};
\node[b] (v1r) at (0.5, 1) {};
\draw (v0l) edge (vl);
\draw (v0l) edge (v1l);
\draw (v0r) edge (vr);
\draw (v0r) edge (v1r);
\end{tikzpicture}
\qquad\qquad
\begin{tikzpicture}
\node[b] (v0) at (0, 0) {};
\node[b] (v1) at (0, 1) {};
\node (v) at (0, -1) {};
\draw (v0) edge (v);
\draw (v0) edge (v1);
\end{tikzpicture}
\caption{Differential on $\Br_{\cC}$.}
\label{fig:brdifferential}
\end{figure}

Consider a morphism of operads
\begin{equation}
\Omega \cC\rightarrow \Br_{\cC}
\label{eq:braceunderlying}
\end{equation}
defined in the following way. The operad $\Omega \cC$ is freely generated by $\overline{\cC}[-1]$ and we send a generator $\s x\in \overline{\cC}[-1]$ to the first corolla shown in Figure \ref{fig:brgenerating} with the internal vertex labeled by $x$.

To see that the morphism \eqref{eq:braceunderlying} is compatible with the differential, note that the differential applied to the first tree in Figure \ref{fig:brgenerating} is equal to the sum of trees of the form
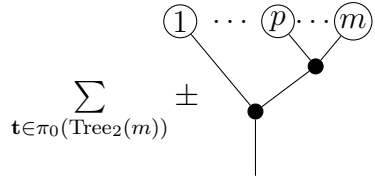
\[
\begin{tikzpicture}
\node[b] (v0) at (0, 0) {};
\node (v) at (0, -1) {};
\node[b] (v1) at (0.8, 0.6) {};
\node[w] (v2) at (-1, 1.2) {$1$};
\node at (-0.3, 1.2) {\dots};
\node[w] (v3) at (0.3, 1.2) {$p$};
\node at (0.8, 1.2) {\dots};
\node[w] (v4) at (1.3, 1.2) {$m$};
\node at (-2, 0) {$\sum\limits_{\t\in\pi_0(\Tree_2(m))}\pm$};
\draw (v0) edge (v);
\draw (v0) edge (v1);
\draw (v0) edge (v2);
\draw (v1) edge (v3);
\draw (v1) edge (v4);
\end{tikzpicture}
\]
where the labels of the two internal vertices are obtained by applying the comultiplication in the cooperad $\cC$ to the original label which is exactly the image of the cobar differential on $\Omega \cC$.

Suppose now $A$ is a $\Br_{\cC}$-algebra. Applying the forgetful morphism \eqref{eq:braceunderlying} we get an $\Omega \cC$-algebra structure on $A$ and hence a differential on the cofree $\cC^{\cu}$-coalgebra $\cC^{\cu}(A)$. Moreover, we get an associative product on $\cC^{\cu}(A)$ defined in the following way. A multiplication
\[\cC^{\cu}(A)\otimes \cC^{\cu}(A)\rightarrow \cC^{\cu}(A)\]
of cofree conilpotent $\cC$-coalgebras is uniquely determined by the projection to the cogenerators
\[\cC^{\cu}(A)\otimes \cC^{\cu}(A)\rightarrow A,\]
i.e. by morphisms
\[\cC^{\cu}(l)\otimes A^{\otimes l}\otimes \cC^{\cu}(m)\otimes A^{\otimes m}\rightarrow A.\]
We let the morphisms with $l=1$ be given by the second corollas in Figure \ref{fig:brgenerating} and those with $l\neq 1$ are defined to be zero. The unit is defined to be the inclusion of $k$ into the second summand of $\cC^{\cu}(A)\cong \cC(A)\oplus k$.

\begin{remark}
The multiplication
\[\cC^{\cu}(l)\otimes A^{\otimes l} \otimes \cC^{\cu}(m)\otimes A^{\otimes m}\rightarrow \cC^{\cu}(A)\]
is given by the composition
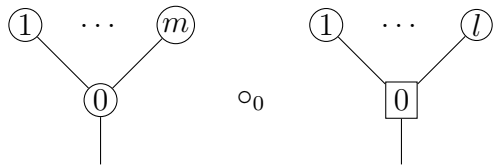
\[
\begin{tikzpicture}
\node[w] (v0l) at (-2, 0) {$0$};
\node (vl) at (-2, -1) {};
\node[w] (v1l) at (-3, 1) {$1$};
\node at (-2, 1) {\dots};
\node[w] (v2l) at (-1, 1) {$m$};
\node at (0, 0) {$\circ_0$};
\node[s] (v0r) at (2, 0) {$0$};
\node (vr) at (2, -1) {};
\node[w] (v1r) at (1, 1) {$1$};
\node at (2, 1) {\dots};
\node[w] (v2r) at (3, 1) {$l$};
\draw (v0l) edge (vl);
\draw (v0l) edge (v1l);
\draw (v0l) edge (v2l);
\draw (v0r) edge (vr);
\draw (v0r) edge (v1r);
\draw (v0r) edge (v2r);
\end{tikzpicture}
\]
In the composition the $\cC^{\cu}$-label of the square vertex is the label of the output in $\cC^{\cu}(A)$ and the number of incoming edges at the square vertex is the arity of the operation in $\cC^{\cu}(A)$.
\label{rmk:koszulmultiplication}
\end{remark}

\begin{prop}
Let $A$ be a $\Br_{\cC}$-algebra. Thus defined multiplication defines on $\cC^{\cu}(A)$ a structure of a dg associative algebra which is compatible with the $\cC^{\cu}$-coalgebra structure.
\label{prop:bracekoszul}
\end{prop}
\begin{proof}
We refer to \cite[Proposition 3.3]{Sa1} for a detailed proof in the case $\cC=\coAss$.

By construction the product is compatible with the $\cC^{\cu}$-coalgebra structure, so we just need to check the associativity of the product and the fact that the differential $\d$ on $\cC^{\cu}(A)$ is a derivation of the product. To check these axioms, it is enough to check that these hold after projections to the cogenerators $A$.

\begin{itemize}
\item (Associativity). The only nontrivial equation is expressed by the commutative diagram
\[
\xymatrix{
(A\otimes \cC^{\cu}(l)\otimes A^{\otimes l})\otimes \cC^{\cu}(m)\otimes A^{\otimes m} \ar[dd] \ar^{\sim}[r] & A\otimes (\cC^{\cu}(l)\otimes A^{\otimes l}\otimes \cC^{\cu}(m)\otimes A^{\otimes m}) \ar[d] \\
& A\otimes \cC^{\cu}(A) \ar[d] \\
A\otimes \cC^{\cu}(m)\otimes A^{\otimes m} \ar[r] & A
}
\]

We denote by $x$ the element of the first $A$ factor, by $y$ elements of $A^{\otimes l}$ and by $z$ elements of $A^{\otimes m}$. Then the composition along the bottom-left corner is given by
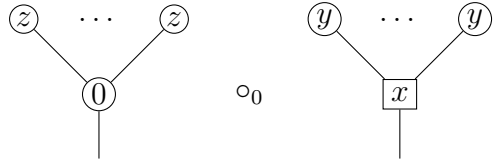
\[
\begin{tikzpicture}
\node[w] (v0l) at (-2, 0) {$0$};
\node (vl) at (-2, -1) {};
\node[w] (v1l) at (-3, 1) {$z$};
\node at (-2, 1) {\dots};
\node[w] (v2l) at (-1, 1) {$z$};
\node at (0, 0) {$\circ_0$};
\node[s] (v0r) at (2, 0) {$x$};
\node (vr) at (2, -1) {};
\node[w] (v1r) at (3, 1) {$y$};
\node at (2, 1) {\dots};
\node[w] (v2r) at (1, 1) {$y$};
\draw (v0l) edge (vl);
\draw (v0l) edge (v1l);
\draw (v0l) edge (v2l);
\draw (v0r) edge (vr);
\draw (v0r) edge (v1r);
\draw (v0r) edge (v2r);
\end{tikzpicture}
\]
which coincides with the composition along the top-right corner following Remark \ref{rmk:koszulmultiplication}.

\item (Derivation). The differential on $\cC^{\cu}(A)$ is given by the sum
\[
\begin{tikzpicture}
\node at (-2.8, 0) {$\d$};
\node[s] (v0l) at (-2, 0) {$0$};
\node (vl) at (-2, -1) {};
\node[w] (v1l) at (-3, 1) {$1$};
\node at (-2, 1) {\dots};
\node[w] (v2l) at (-1, 1) {$m$};
\draw (v0l) edge (vl);
\draw (v0l) edge (v1l);
\draw (v0l) edge (v2l);

\node at (0, 0) {$=\sum\limits_{\t\in\pi_0(\Tree_2(m))}\pm$};

\node[s] (v0r) at (2, 0) {$0$};
\node (vr) at (2, -1) {};
\node[b] (v1r) at (2.8, 0.6) {};
\node[w] (v2r) at (1, 1.2) {$1$};
\node at (1.7, 1.2) {\dots};
\node[w] (v3r) at (2.3, 1.2) {$p$};
\node at (2.8, 1.2) {\dots};
\node[w] (v4r) at (3.3, 1.2) {$m$};
\draw (v0r) edge (vr);
\draw (v0r) edge (v1r);
\draw (v0r) edge (v2r);
\draw (v1r) edge (v3r);
\draw (v1r) edge (v4r);
\end{tikzpicture}
\]

The compatibility of the differential and the multiplication then immediately follows from the compatibility of the differential and the composition in the operad $\Br_{\cC}$ since the multiplication on $\cC^{\cu}(A)$ is defined in terms of the composition.
\end{itemize}
\end{proof}

This proposition motivates the following definition.
\begin{defn}
\label{def:infinitybracemor}
An \emph{$\infty$-morphism of $\Br_{\cC}$-algebras} $A\rightarrow B$ is a morphism of dg associative algebras $\cC^{\cu}(A)\rightarrow \cC^{\cu}(B)$ compatible with the $\cC^{\cu}$-coalgebra structures.
\end{defn}

Suppose $\cC^{\cu}(A)\rightarrow \cC^{\cu}(B)$ is an $\infty$-morphism of $\Br_{\cC}$-algebras. Using the canonical unit morphism $\coComm\rightarrow \cC^{\cu}$ we obtain a morphism $\Sym(A)\rightarrow \Sym(B)$ which preserves the cocommutative comultiplication and the multiplication. In particular, it induces a \emph{strict} morphism of Lie algebras $A\rightarrow B$ after passing to primitive elements.

Let us now consider the case when $\cC^{\cu}$ is a graded Hopf cooperad, i.e. a cooperad in graded commutative dg algebras. The operad $\Br_{\cC}$ inherits the grading given by the sum of all weights of labels in $\cC^{\cu}$. Now recall that $\Br_{\cC}$ acts on $\Conv^0(\cC^{\cu}\{n\}, A)$ for any $n$. If $\cC^{\cu}\{n\}$ has a structure of a graded module over $\cC^{\cu}$, then the natural grading on the convolution algebra $\Conv^0(\cC^{\cu}\{n\}, A)$ makes it into a graded algebra over the graded operad $\Br_{\cC}$.

\begin{remark}
Note that we do not assume that the grading on $\cC^{\cu}\{n\}$ coincides with the grading on $\cC^{\cu}$. In fact, in our main example of polyvector fields the two gradings are distinct.
\end{remark}

Fix the number $n$.
\begin{defn}
Let $A$ be an $\Omega(\cC\{n\})$-algebra whose structure is determined by a Maurer--Cartan element $f\in\Conv(\cC\{n\}; A)$. The \emph{center} of $A$ is the shifted convolution algebra
\[\bZ(A) = \Conv^0(\cC^{\cu}\{n\}; A)[-n]\]
twisted by the Maurer--Cartan element $f$.
\label{def:center}
\end{defn}

Note that by construction the center $\bZ(A)$ is naturally an algebra over $\Br_\cC\{n\}$.

\subsection{Relative brace construction}
\label{sect:relativebrace}

Let us now describe a relative version of these constructions. Consider a pair of complexes $A,B$. Then we can define a relative convolution algebra to be the complex
\[\Conv^0(\cC; A, B) = \Hom_k(\cC(A), B).\]
We are now going to introduce a certain algebraic structure on the triple \[(\Conv^0(\cC^{\cu}; A), \Conv^0(\cC; A, B), \Conv^0(\cC^{\cu}; B))\] which generalizes the pre-Lie structure on the convolution algebra.

Consider the set of colors $\cV=\{\cA\rightarrow \cA, \cA\rightarrow \cB, \cB\rightarrow \cB\}$. We introduce a $\cV$-colored version of the pre-Lie operad denoted by $\preLie^{\rightarrow}$ in the following way. Operations of $\preLie^{\rightarrow}$ are parametrized by rooted trees with edges of two types: of type $\cA$ that we denote by solid lines and of type $\cB$ that we denote by dashed lines. We disallow any vertices which have incoming edges of different types or those that have incoming edges of type $\cB$ but an outgoing edge of type $\cA$. A color of the vertex is determined by the type of inputs and outputs and we denote it as e.g. $\cA\rightarrow \cB$. To resolve ambiguities, we draw incoming edges to leaves (recall that in the case of the ordinary pre-Lie operad we de not draw incoming edges to leaves following \cite{CL}). One can read off the arity of the operation parametrized by a rooted tree in the following way. Each vertex has a color determined by incoming and outgoing edges and so does the whole graph and this determines the arity. See Figure \ref{fig:coloredprelie} for an example. The operadic composition is given by grafting trees exactly in the same way as in the case of the pre-Lie operad.

\begin{figure}
\begin{tikzpicture}
\node[w] (v0) at (0, 0) {$1$};
\node (v) at (0, -1) {};
\node[w] (v1) at (-1, 1) {$2$};
\node (v11) at (-1, 1.7) {};
\node[w] (v2) at (1, 1) {$3$};
\node (v21) at (1, 1.7) {};
\draw[dashed] (v0) edge (v);
\draw (v0) edge (v1);
\draw (v1) edge (v11);
\draw (v0) edge (v2);
\draw (v2) edge (v21);
\end{tikzpicture}
\caption{An example of an operation in $\preLie^{\rightarrow}$ of arity $((\cA\rightarrow \cA)^{\otimes 2}\otimes (\cA\rightarrow \cB), (\cA\rightarrow \cB))$.}
\label{fig:coloredprelie}
\end{figure}
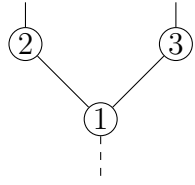

The colored operad $\preLie^{\rightarrow}$ acts on the triple \[(\Conv^0(\cC^{\cu}; A), \Conv^0(\cC; A, B), \Conv^0(\cC^{\cu}; B))\] exactly in the same way as in the case of the usual pre-Lie operad. Namely, given a rooted tree we substitute elements of the convolution algebras into vertices based on colors:
\begin{itemize}
\item if a vertex has color $\cA\rightarrow \cA$, we substitute an element of $\Conv^0(\cC^{\cu}; A)$,

\item if a vertex has color $\cA\rightarrow \cB$, we substitute an element of $\Conv^0(\cC; A, B)$,

\item if a vertex has color $\cB\rightarrow \cB$, we substitute an element of $\Conv^0(\cC^{\cu}; B)$.
\end{itemize}
After such a substitution one reads off the result by composing the morphisms using the pattern given by the rooted tree.

Given a $\preLie^{\rightarrow}$-algebra $(C_1, C_2, C_3)$ one has a natural $L_\infty$ structure on
\[C_1\oplus C_2[-1]\oplus C_3\]
given by expressions in Figure \ref{fig:coloredprelietolie}. Here the first three trees give rise to ordinary Lie brackets and the last tree gives rise to an $L_\infty$ operation.

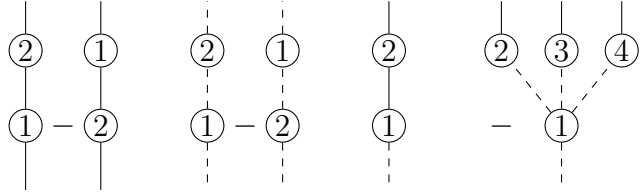
\begin{figure}
\begin{tikzpicture}
\node[w] (v0l) at (-0.5, 0) {$1$};
\node (vl) at (-0.5, -1) {};
\node[w] (v1l) at (-0.5, 1) {$2$};
\node (v2l) at (-0.5, 1.8) {};
\node at (0, 0) {$-$};
\node[w] (v0r) at (0.5, 0) {$2$};
\node (vr) at (0.5, -1) {};
\node[w] (v1r) at (0.5, 1) {$1$};
\node (v2r) at (0.5, 1.8) {};
\draw (v0l) edge (vl);
\draw (v0l) edge (v1l);
\draw (v1l) edge (v2l);
\draw (v0r) edge (vr);
\draw (v0r) edge (v1r);
\draw (v1r) edge (v2r);
\end{tikzpicture}
\qquad
\begin{tikzpicture}
\node[w] (v0l) at (-0.5, 0) {$1$};
\node (vl) at (-0.5, -1) {};
\node[w] (v1l) at (-0.5, 1) {$2$};
\node (v2l) at (-0.5, 1.8) {};
\node at (0, 0) {$-$};
\node[w] (v0r) at (0.5, 0) {$2$};
\node (vr) at (0.5, -1) {};
\node[w] (v1r) at (0.5, 1) {$1$};
\node (v2r) at (0.5, 1.8) {};
\draw[dashed] (v0l) edge (vl);
\draw[dashed] (v0l) edge (v1l);
\draw[dashed] (v1l) edge (v2l);
\draw[dashed] (v0r) edge (vr);
\draw[dashed] (v0r) edge (v1r);
\draw[dashed] (v1r) edge (v2r);
\end{tikzpicture}
\qquad
\begin{tikzpicture}
\node[w] (v0) at (0, 0) {$1$};
\node (v) at (0, -1) {};
\node[w] (v1) at (0, 1) {$2$};
\node (v2) at (0, 1.8) {};
\draw[dashed] (v0) edge (v);
\draw (v0) edge (v1);
\draw (v1) edge (v2);
\end{tikzpicture}
\qquad
\begin{tikzpicture}
\node[w] (v0) at (0, 0) {$1$};
\node (v) at (0, -1) {};
\node[w] (v1) at (-0.8, 1) {$2$};
\node (v11) at (-0.8, 1.8) {};
\node[w] (v2) at (0, 1) {$3$};
\node (v21) at (0, 1.8) {};
\node[w] (v3) at (0.8, 1) {$4$};
\node (v31) at (0.8, 1.8) {};
\node at (-0.8, 0) {$-$};
\draw[dashed] (v0) edge (v);
\draw[dashed] (v0) edge (v1);
\draw (v1) edge (v11);
\draw[dashed] (v0) edge (v2);
\draw (v2) edge (v21);
\draw[dashed] (v0) edge (v3);
\draw (v3) edge (v31);
\end{tikzpicture}
\caption{$L_\infty$ brackets on a $\preLie^{\rightarrow}$ algebra.}
\label{fig:coloredprelietolie}
\end{figure}

Since $\cC^{\cu}$ is a Hopf cooperad, one can similarly define a $\cV$-colored operad $\preLie^{\rightarrow}_{\cC}$ whose operations are parametrized by rooted trees with edges of two types as in the case of $\preLie^{\rightarrow}$ and whose vertices are parametrized by elements of $\cC$. As before, the triple $(\Conv^0(\cC^{\cu}; A), \Conv^0(\cC; A, B), \Conv^0(\cC^{\cu}; B))$ is an algebra over the colored operad $\preLie^{\rightarrow}_{\cC}$.

Using the forgetful map from $\preLie^{\rightarrow}_{\cC}$-algebras to $L_\infty$-algebras, one can apply the general formalism of twistings of \cite{DW1} to construct the colored operad $\Br_{\cC}^{\rightarrow}$ whose operations are parametrized by rooted trees with dashed and solid edges and external and internal vertices as before.

Suppose \[f=(f_1, f_2, f_3)\in (\Conv^0(\cC^{\cu}; A), \Conv^0(\cC; A, B), \Conv^0(\cC^{\cu}; B))\] is a Maurer--Cartan element in the underlying $L_\infty$-algebra. We can twist the differential on $\Conv^0(\cC^{\cu}; A)$ using $f_1$, we can twist the differential on $\Conv^0(\cC^{\cu}; B)$ using $f_3$ and we can twist the differential on $\Conv^0(\cC; A, B)$ using all three elements. Then as before the triple
\[(\Conv^0_{f_1}(\cC^{\cu}; A), \Conv_f^0(\cC; A, B), \Conv_{f_3}^0(\cC^{\cu}; B))\]
becomes an algebra over the colored operad $\Br_{\cC}^{\rightarrow}$ if we assign $f$ to internal vertices.

Suppose now $(C_1, C_2, C_3)$ is any $\Br_{\cC}^{\rightarrow}$-algebra. In particular, $C_1$ and $C_3$ are $\Br_{\cC}$-algebras. One naturally has an $\Omega\cC$-algebra structure on $C_2$ given by sending an element $\s x\in \overline{\cC}[-1]$ to the first corolla shown in Figure \ref{fig:bracemorphism} where the internal vertex is labeled by the element $x$. In particular, $\Conv(\cC; C_2)$ has a Maurer--Cartan element that we denote by $f$. Note that this should not be confused with the previous occurrence of Maurer--Cartan elements in $\Conv^0(\cC^{\cu}; B)$.

\begin{figure}
\begin{tikzpicture}
\node[b] (v0l) at (-5, 0) {};
\node (vl) at (-5, -1) {};
\node[w] (v1l) at (-6, 1) {$1$};
\node (v11l) at (-6, 1.8) {};
\node at (-5, 1) {\dots};
\node[w] (v2l) at (-4, 1) {$m$};
\node (v21l) at (-4, 1.8) {};
\draw[dashed] (v0l) edge (vl);
\draw[dashed] (v0l) edge (v1l);
\draw[dashed] (v0l) edge (v2l);
\draw (v1l) edge (v11l);
\draw (v2l) edge (v21l);

\node[w] (v0) at (0, 0) {$0$};
\node (v) at (0, -1) {};
\node[w] (v1) at (-1, 1) {$1$};
\node (v11) at (-1, 1.8) {};
\node at (0, 1) {\dots};
\node[w] (v2) at (1, 1) {$m$};
\node (v21) at (1, 1.8) {};
\draw[dashed] (v0) edge (v);
\draw[dashed] (v0) edge (v1);
\draw[dashed] (v0) edge (v2);
\draw (v1) edge (v11);
\draw (v2) edge (v21);

\node[w] (v0r) at (5, 0) {$0$};
\node (vr) at (5, -1) {};
\node[w] (v1r) at (4, 1) {$1$};
\node (v11r) at (4, 1.8) {};
\node at (5, 1) {\dots};
\node[w] (v2r) at (6, 1) {$m$};
\node (v21r) at (6, 1.8) {};
\draw[dashed] (v0r) edge (vr);
\draw (v0r) edge (v1r);
\draw (v0r) edge (v2r);
\draw (v1r) edge (v11r);
\draw (v2r) edge (v21r);
\end{tikzpicture}
\caption{An $\Omega\cC$-structure on $C_2$, the morphism $C_3\rightarrow \Conv_f(\cC; C_2)$ and the $\infty$-morphism $C_1\rightarrow \Conv_f(\cC; C_2)$ respectively.}
\label{fig:bracemorphism}
\end{figure}
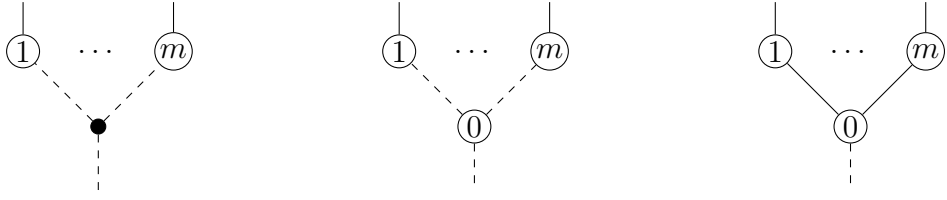

We have a morphism
\begin{equation}
C_3\rightarrow \Conv_f^0(\cC^{\cu}; C_2)
\label{eq:bracemorphism1}
\end{equation}
defined in the following way. The second corolla shown in Figure \ref{fig:bracemorphism} defines a morphism \[\cC^{\cu}(m)\otimes C_1\otimes C_2^{\otimes m}\rightarrow C_2,\] where the element of $\cC^{\cu}(m)$ labels vertex $0$. By adjunction this gives the required morphism.

Similarly, we define the morphism
\begin{equation}
\cC^{\cu}(C_1)\longrightarrow \Conv_f^0(\cC^{\cu}; C_2)
\label{eq:bracemorphism2}
\end{equation}
in the following way. The third corolla shown in Figure \ref{fig:bracemorphism} defines a morphism
\[\cC^{\cu}(m)\otimes C_1^{\otimes m}\otimes C_2\longrightarrow C_2\]
which by adjunction gives the morphism
\[\cC^{\cu}(C_1)\longrightarrow \Hom(C_2, C_2)\rightarrow \Conv_f^0(\cC^{\cu}; C_2).\]

\begin{remark}
One can give the following pictorial representation of the morphism \[\cC^{\cu}(m)\otimes C_1^{\otimes m}\longrightarrow \cC^{\cu}(\Conv_f^0(\cC^{\cu}; C_2)):\]
\[
\begin{tikzpicture}
\node[s] (v0l) at (-2, 0) {$\ $};
\node (vl) at (-2, -1) {};
\node[w] (v1l) at (-3, 1) {$1$};
\node (v11l) at (-3, 1.8) {};
\node at (-2, 1) {\dots};
\node[w] (v3l) at (-1, 1) {$m$};
\node (v31l) at (-1, 1.8) {};
\draw (v0l) edge (vl);
\draw (v0l) edge (v1l);
\draw (v1l) edge (v11l);
\draw (v0l) edge (v3l);
\draw (v3l) edge (v31l);

\node at (0, 0) {$\mapsto\sum$};

\node[s] (v0r) at (2, 0) {$\ $};
\node (vr) at (2, -1) {};
\node[w] (v1r) at (1, 1) {$\ $};
\node[w] (v11r) at (0.7, 1.8) {$1$};
\node (v111r) at (0.7, 2.6) {};
\node[w] (v12r) at (1.3, 1.8) {$2$};
\node at (2.2, 1.8) {\dots};
\node (v121r) at (1.3, 2.6) {};
\node[w] (v2r) at (3, 1) {$\ $};
\node[w] (v21r) at (3, 1.8) {$m$};
\node (v211r) at (3, 2.6) {};
\draw[dashed] (v0r) edge (vr);
\draw[dashed] (v0r) edge (v1r);
\draw (v1r) edge (v11r);
\draw (v11r) edge (v111r);
\draw (v1r) edge (v12r);
\draw (v12r) edge (v121r);
\draw[dashed] (v0r) edge (v2r);
\draw (v2r) edge (v21r);
\draw (v21r) edge (v211r);
\end{tikzpicture}
\]

Here the label of the rectangle on the left is the element $c\in\cC^{\cu}(m)$. The labels of the unmarked vertices on the right are given by applying the coproduct in the cooperad $\cC^{\cu}$ to $c$ with respect to the corresponding picthfork.
\label{rmk:koszulbracemorphism}
\end{remark}

\begin{prop}
Let $(C_1, C_2, C_3)$ be a $\Br_{\cC}^{\rightarrow}$-algebra. Then the morphism \eqref{eq:bracemorphism1}
\[C_3\rightarrow \Conv_f^0(\cC^{\cu}; C_2)\]
is a morphism of $\Br_{\cC}$-algebras.

Similarly, the morphism \eqref{eq:bracemorphism2}
\[C_1\rightarrow \Conv_f^0(\cC^{\cu}; C_2)\]
is an $\infty$-morphism of $\Br_{\cC}$-algebras.
\label{prop:bracemorphisms}
\end{prop}
\begin{proof}
For the first statement we have to show that the diagram
\[
\xymatrix{
\Br_{\cC}(m)\otimes C_3^{\otimes m} \ar[r] \ar[d] & \Br_{\cC}(m)\otimes \Conv_f^0(\cC^{\cu}; C_2)^{\otimes m} \ar[d] \\
C_3 \ar[r] & \Conv_f^0(\cC^{\cu}; C_2)
}
\]
commutes. Pick $\t\in\Br_{\cC}(m)$ and $x_1, \dots, x_m\in C_3$. The composition along the bottom-left corner is given by applying composition using the pattern given by the trees
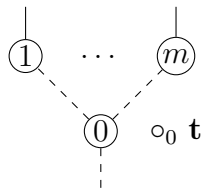
\[
\begin{tikzpicture}
\node[w] (v0) at (0, 0) {$0$};
\node (v) at (0, -1) {};
\node[w] (v1) at (-1, 1) {$1$};
\node (v11) at (-1, 1.8) {};
\node at (0, 1) {\dots};
\node[w] (v2) at (1, 1) {$m$};
\node (v21) at (1, 1.8) {};
\node at (1, 0) {$\circ_0\ \t$};
\draw[dashed] (v0) edge (v);
\draw[dashed] (v0) edge (v1);
\draw[dashed] (v0) edge (v2);
\draw (v1) edge (v11);
\draw (v2) edge (v21);
\end{tikzpicture}
\]
where the vertices of $\t$ are labeled by the elements $x_i$. Similarly, the composition along the top-right corner is given by the sum over numbers $n_1, \dots, n_m$ of trees given by attaching $n_i$ incoming edges at vertex $i$ of $\t$. The two expressions obviously coincide. For instance, in the composition
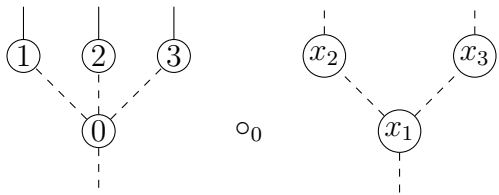
\[
\begin{tikzpicture}
\node[w] (v0l) at (-2, 0) {$0$};
\node (vl) at (-2, -1) {};
\node[w] (v1l) at (-3, 1) {$1$};
\node (v11l) at (-3, 1.8) {};
\node[w] (v2l) at (-2, 1) {$2$};
\node (v21l) at (-2, 1.8) {};
\node[w] (v3l) at (-1, 1) {$3$};
\node (v31l) at (-1, 1.8) {};
\draw[dashed] (v0l) edge (vl);
\draw[dashed] (v0l) edge (v1l);
\draw[dashed] (v0l) edge (v2l);
\draw[dashed] (v0l) edge (v3l);
\draw (v1l) edge (v11l);
\draw (v2l) edge (v21l);
\draw (v3l) edge (v31l);

\node at (0, 0) {$\circ_0$};

\node[w] (v0r) at (2, 0) {$x_1$};
\node (vr) at (2, -1) {};
\node[w] (v1r) at (1, 1) {$x_2$};
\node (v11r) at (1, 1.8) {};
\node[w] (v2r) at (3, 1) {$x_3$};
\node (v21r) at (3, 1.8) {};
\draw[dashed] (v0r) edge (vr);
\draw[dashed] (v0r) edge (v1r);
\draw[dashed] (v1r) edge (v11r);
\draw[dashed] (v0r) edge (v2r);
\draw[dashed] (v2r) edge (v21r);
\end{tikzpicture}
\]
the term with $n_1 = 1$, $n_2 = 2$ and $n_3 = 0$ is
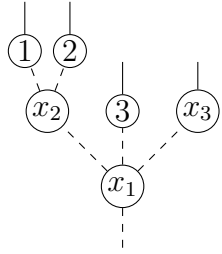
\[
\begin{tikzpicture}
\node[w] (v0) at (0, 0) {$x_1$};
\node (v) at (0, -1) {};
\node[w] (v01) at (0, 1) {$3$};
\node (v011) at (0, 1.8) {};
\node[w] (v1) at (-1, 1) {$x_2$};
\node[w] (v11) at (-1.3, 1.8) {$1$};
\node (v111) at (-1.3, 2.6) {};
\node[w] (v12) at (-0.7, 1.8) {$2$};
\node (v121) at (-0.7, 2.6) {};
\node[w] (v2) at (1, 1) {$x_3$};
\node (v21) at (1, 1.8) {};
\draw[dashed] (v0) edge (v);
\draw[dashed] (v0) edge (v1);
\draw[dashed] (v1) edge (v11);
\draw[dashed] (v1) edge (v12);
\draw (v11) edge (v111);
\draw (v12) edge (v121);
\draw[dashed] (v0) edge (v01);
\draw (v01) edge (v011);
\draw[dashed] (v0) edge (v2);
\draw (v2) edge (v21);
\end{tikzpicture}
\]

For the second statement we have to check that
\[\cC^{\cu}(C_1)\rightarrow \cC^{\cu}(\Conv_f^0(\cC^{\cu}; C_2))\]
is a morphism compatible with the differentials and multiplications. The computation is similar to the proof of the first statement and uses the description of the morphism given in Remark \ref{rmk:koszulbracemorphism}.
\end{proof}

\begin{remark}
If $C$ is a $\Br_{\cC}$-algebra and $D$ is an $\Omega \cC$-algebra, then an $\infty$-morphism of $\Br_{\cC}$-algebras $C\rightarrow D$ is essentially the same as the notion of a brace module from \cite[Definition 3.2]{Sa1} when $\cC=\coAss$. In this case an analog of the second statement in the previous proposition is \cite[Proposition 4.2]{Sa1}.
\end{remark}

\subsection{Swiss-cheese construction}
\label{sect:SCalgebras}

Recall from \cite{Th} that a Swiss-cheese algebra consists of an $\bE_2$-algebra $A$, an $\bE_1$-algebra $B$ and an $\bE_2$-morphism $A\rightarrow \mathrm{HH}^\bullet(B)$ to the Hochschild cohomology of $B$. A model of the $\bE_2$ operad is given by the brace operad which can be obtained in our notation as $\Br_{\coAss}\{1\}$, the brace construction on the Hopf cooperad of coassociative coalgebras. In this section we construct a colored operad $\SC(\cC_1, \cC_2)$ generalizing the Swiss-cheese operad using the brace construction. An algebra over $\SC(\cC_1, \cC_2)$ will be an $\Omega\cC_1$-algebra $A$, an $\Omega(\cC_2\{n\})$-algebra $B$ and an $\infty$-morphism of $\Omega\cC_1$-algebras $A\rightarrow \bZ(B)$. The construction of the colored operad $\SC(\cC_1, \cC_2)$ will be modeled after the resolution of the operad controlling morphisms of $\cO$-algebras constructed in \cite[Section 2]{Ma}.

Fix a number $n$. Let $\cC_1$ be a cooperad and $\cC_2$ a cooperad with a Hopf counital structure together with an operad morphism \[F\colon \Omega\cC_1\rightarrow \Br_{\cC_2}\{n\}.\]
From this data we define a semi-free colored operad $\SC(\cC_1, \cC_2)$ in the following way.

The set of colors of $\SC(\cC_1, \cC_2)$ is $\{\cA, \cB\}$. The operad is semi-free on the colored symmetric sequence $P(\cC_1, \cC_2)$ whose nonzero elements are
\begin{align*}
&P(\cC_1, \cC_2)(\cA^{\otimes m}, \cA) = \overline{\cC_1}(m) \\
&P(\cC_1, \cC_2)(\cB^{\otimes l}, \cB) = \overline{\cC_2}\{n\}(l) \\
&P(\cC_1, \cC_2)(\cA^{\otimes m}\otimes \cB^{\otimes l}, \cB) = \cC_1(m)\otimes \cC_2^{\cu}\{n\}(l)[n+1].
\end{align*}

The colored operad $\Free(P(\cC_1, \cC_2)[-1])$ has operations parametrized by trees with edges of two types: those of color $\cA$ that we denote by solid lines and those of color $\cB$ that we denote by dashed lines. The vertices of the trees are labeled by generating operations in $P(\cC_1, \cC_2)$. We define a differential on $\Free(P(\cC_1, \cC_2)[-1])$ in the following way. The differentials in arities $(\cA^{\otimes -}, \cA)$ and $(\cB^{\otimes -}, \cB)$ are the usual cobar differentials \eqref{eq:cobardifferential}. The differential on an element $\s^{-n} X\otimes Y$ for $X\in \cC_1(m)$ and $Y\in\cC_2^{\cu}\{n\}(l)$ has four components:
\begin{enumerate}
\item
\[\d_1(\s^{-n} X\otimes Y) = (-1)^n \s^{-n} \d_1 X\otimes Y + (-1)^{n+|X|} \s^{-n} X\otimes \d_1 Y\]
where $\d_1$ are the internal differentials on the complexes $\cC_1(m)$ and $\cC_2^{\cu}\{n\}(l)$.

\item
\[\d_2(\s^{-n} X\otimes Y) = (-1)^n (\s^{-n-1} \otimes 1) \sum_{\t\in\pi_0(\Tree_2(m))} (\s\otimes \s)(\t, \Delta_\t(X))\circ_0 Y,\]
where we use the following notation. Let $\Delta_\t(X) = X_{(0)}\otimes X_{(1)}$ with $X_{(0)}$ the label of the root. We denote by $(\t, \Delta_\t(X))\circ_0 Y$ the tree $\t$ with additional $l$ dashed incoming edges at the root which is labeled by $X_{(0)}\otimes Y$. The right-hand side consists of a composition of an operation in $P(\cA^{\otimes -}, \cA)$ and $P(\cA^{\otimes -}\otimes \cB^{\otimes -}, \cB)$. See Figure \ref{fig:SCd2} for an example.

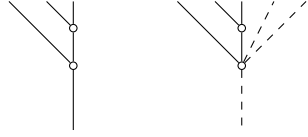
\begin{figure}[h]
\begin{tikzpicture}
\node[w] (v0) at (0, 0) {};
\node (root) at (0, -1) {};
\node[w] (v1) at (0, 0.5) {};
\node (v01) at (-1, 1) {};
\node (v11) at (-0.5, 1) {};
\node (v12) at (0, 1) {};
\draw (v0) edge (root);
\draw (v0) edge (v1);
\draw (v0) edge (v01);
\draw (v1) edge (v11);
\draw (v1) edge (v12);
\end{tikzpicture}
\qquad
\begin{tikzpicture}
\node[w] (v0) at (0, 0) {};
\node (root) at (0, -1) {};
\node[w] (v1) at (0, 0.5) {};
\node (v01) at (-1, 1) {};
\node (v02) at (0.5, 1) {};
\node (v03) at (1, 1) {};
\node (v11) at (-0.5, 1) {};
\node (v12) at (0, 1) {};
\draw[dashed] (v0) edge (root);
\draw (v0) edge (v1);
\draw (v0) edge (v01);
\draw[dashed] (v0) edge (v02);
\draw[dashed] (v0) edge (v03);
\draw (v1) edge (v11);
\draw (v1) edge (v12);
\end{tikzpicture}
\caption{A tree $\t$ and $\t\circ_0 Y$ with $l=2$.}
\label{fig:SCd2}
\end{figure}

\item
\begin{align*}
\d_3(\s^{-n} X\otimes Y) &= (-1)^n (\s^{-n-1}\otimes 1) X\circ_0 \sum_{\t\in\pi_0(\Tree_2(l))} (\s\otimes \s)(\t, \Delta_\t(Y)) \\
&+ (-1)^n (1 \otimes \s^{-n-1}) X\circ_1 \sum_{\t\in\pi_0(\Tree_2(l))} (\s\otimes \s)(\t, \Delta_\t(Y)),
\end{align*}
where we use the following notation. Let $\Delta_\t(Y) = Y_{(0)}\otimes Y_{(1)}$ with $Y_{(0)}$ the label of the root. We denote by $X\circ_0 (\t, \Delta_\t(Y))$ the tree $\t$ with additional $m$ solid incoming edges at the root which is labeled by $X\otimes Y_{(0)}$. Similarly, $X\circ_1 (\t, \Delta_\t(Y))$ is the tree $\t$ with additional $m$ solid incoming edges at the other node which is labeled by $X\otimes Y_{(1)}$. See figure \ref{fig:SCd3} for an example.

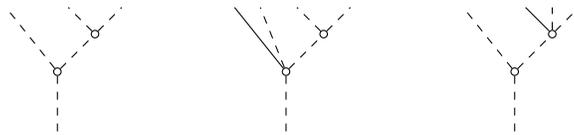
\begin{figure}[h]
\begin{tikzpicture}
\node[w] (v0) at (0, 0) {};
\node (root) at (0, -1) {};
\node[w] (v1) at (0.5, 0.5) {};
\node (v01) at (-0.8, 1) {};
\node (v11) at (0, 1) {};
\node (v12) at (1, 1) {};
\draw[dashed] (v0) edge (root);
\draw[dashed] (v0) edge (v1);
\draw[dashed] (v0) edge (v01);
\draw[dashed] (v1) edge (v11);
\draw[dashed] (v1) edge (v12);
\end{tikzpicture}
\qquad
\begin{tikzpicture}
\node[w] (v0) at (0, 0) {};
\node (root) at (0, -1) {};
\node[w] (v1) at (0.5, 0.5) {};
\node (v01) at (-0.8, 1) {};
\node (v02) at (-0.4, 1) {};
\node (v11) at (0, 1) {};
\node (v12) at (1, 1) {};
\draw[dashed] (v0) edge (root);
\draw[dashed] (v0) edge (v1);
\draw (v0) edge (v01);
\draw[dashed] (v0) edge (v02);
\draw[dashed] (v1) edge (v11);
\draw[dashed] (v1) edge (v12);
\end{tikzpicture}
\qquad
\begin{tikzpicture}
\node[w] (v0) at (0, 0) {};
\node (root) at (0, -1) {};
\node[w] (v1) at (0.5, 0.5) {};
\node (v01) at (-0.8, 1) {};
\node (v11) at (0, 1) {};
\node (v12) at (0.5, 1) {};
\node (v13) at (1, 1) {};
\draw[dashed] (v0) edge (root);
\draw[dashed] (v0) edge (v1);
\draw[dashed] (v0) edge (v01);
\draw (v1) edge (v11);
\draw[dashed] (v1) edge (v12);
\draw[dashed] (v1) edge (v13);
\end{tikzpicture}
\caption{A tree $\t$, $X\circ_0 \t$ and $X\circ_1 \t$ with $m=1$.}
\label{fig:SCd3}
\end{figure}

\item
\[\d_4(\s^{-n} X\otimes Y) = \sum_r \sum_{\t\in\Isom_\pitchfork(m, r)} F(\t, \Delta_\t(X), Y).\]

Here $F(\t, \Delta_\t(X), Y)$ is defined in the following way. Denote by $X_{(i)}\in \cC_1$ the labels of the vertices in $\Delta_\t(X)$ with $X_{(0)}$ the label of the root. The image of $\s X_{(0)}$ under $F\colon \Omega\cC_1\rightarrow \Br_{\cC_2}\{n\}$ is a rooted tree $F(\s X_{(0)})$ labeled by $r$ elements $Z_{(i)}\in\cC_2^{\cu}$. Consider the composition $\t\circ_0 F(\s X_{(0)})$. We consider the following set of trees $\tilde{\t}$: a tree $\tilde{\t}$ is obtained from $\t\circ_0 F(\s X_{(0)})$ by adding an arbitrary number of incoming dashed edges to vertices so that the total number of incoming dashed edges is $l$. Let us denote by $\tilde{\t}_{dashed}$ the tree obtained from $\tilde{t}$ by erasing all solid edges. We let $\Delta_{\t_{dashed}}(Y) = Y_{(1)}\otimes \dots \otimes Y_{(r)}$. The labelings of vertices of $\tilde{\t}$ are of two kinds: external vertices are labeled by the tensor product $X_{(i)}\otimes Y_{(i)}Z_{(i)}$ where $Y_{(i)}Z_{(i)}$ is the product in the Hopf cooperad $\cC_2^{\cu}$ and they belong to the operations in $P(\cA^{\otimes -}\otimes \cB^{\otimes -}, \cB)$; the internal vertices are simply labeled by elements of $\cC_2$ and they belong to the operations in $P(\cB^{\otimes -}, \cB)$. We refer to Figure \ref{fig:SCd4} for an example. We define $F(\t, \Delta_\t(X), Y)$ to be the sum over all such trees $\tilde{\t}$.

\begin{figure}[h]
\begin{tikzpicture}
\node[w] (v0) at (0, 0) {};
\node (root) at (0, -0.5) {};
\node[w] (v1) at (-1, 0.5) {};
\node[w] (v2) at (0, 0.5) {};
\node[w] (v3) at (1, 0.5) {};
\node (v11) at (-1.2, 1) {};
\node (v12) at (-0.8, 1) {};
\node (v21) at (-0.2, 1) {};
\node (v22) at (0.2, 1) {};
\node (v31) at (0.8, 1) {};
\node (v32) at (1.2, 1) {};
\draw (v0) edge (root);
\draw (v0) edge (v1);
\draw (v0) edge (v2);
\draw (v0) edge (v3);
\draw (v1) edge (v11);
\draw (v1) edge (v12);
\draw (v2) edge (v21);
\draw (v2) edge (v22);
\draw (v3) edge (v31);
\draw (v3) edge (v32);
\end{tikzpicture}
\qquad
\begin{tikzpicture}
\node[b] (v0) at (0, 0) {};
\node (root) at (0, -0.5) {};
\node[w] (v1) at (-0.5, 0.7) {$1$};
\node[w] (v2) at (-0.5, 1.4) {$2$};
\node[w] (v3) at (0.5, 0.7) {$3$};
\draw (v0) edge (root);
\draw (v0) edge (v1);
\draw (v0) edge (v3);
\draw (v1) edge (v2);
\end{tikzpicture}
\qquad
\begin{tikzpicture}
\node[w] (v0) at (0, 0) {};
\node (root) at (0, -0.5) {};
\node[w] (v1) at (-0.7, 0.5) {};
\node[w] (v2) at (-1.4, 1) {};
\node[w] (v3) at (0.7, 0.5) {};
\node (v11) at (-0.9, 1.5) {};
\node (v12) at (-0.7, 1.5) {};
\node (v13) at (-0.5, 1.5) {};
\node (v21) at (-1.6, 1.5) {};
\node (v22) at (-1.2, 1.5) {};
\node (v31) at (0.5, 1.5) {};
\node (v32) at (0.7, 1.5) {};
\node (v33) at (0.9, 1.5) {};
\draw[dashed] (v0) edge (root);
\draw[dashed] (v0) edge (v1);
\draw[dashed] (v0) edge (v3);
\draw[dashed] (v1) edge (v2);
\draw (v1) edge (v11);
\draw (v1) edge (v12);
\draw[dashed] (v1) edge (v13);
\draw (v2) edge (v21);
\draw (v2) edge (v22);
\draw (v3) edge (v31);
\draw (v3) edge (v32);
\draw[dashed] (v3) edge (v33);
\end{tikzpicture}
\qquad
\begin{tikzpicture}
\node[w] (v0) at (0, 0) {};
\node (root) at (0, -0.5) {};
\node[w] (v1) at (-0.7, 0.5) {};
\node[w] (v2) at (-1.4, 1) {};
\node[w] (v3) at (0.7, 0.5) {};
\node (v11) at (-0.7, 1.5) {};
\node (v31) at (0.7, 1.5) {};
\draw[dashed] (v0) edge (root);
\draw[dashed] (v0) edge (v1);
\draw[dashed] (v0) edge (v3);
\draw[dashed] (v1) edge (v2);
\draw[dashed] (v1) edge (v11);
\draw[dashed] (v3) edge (v31);
\end{tikzpicture}
\caption{A pitchfork $\t$, a rooted tree $F(\s X_{(0)})$, an example of $\tilde{\t}$ and $\tilde{\t}_{dashed}$.}
\label{fig:SCd4}
\end{figure}
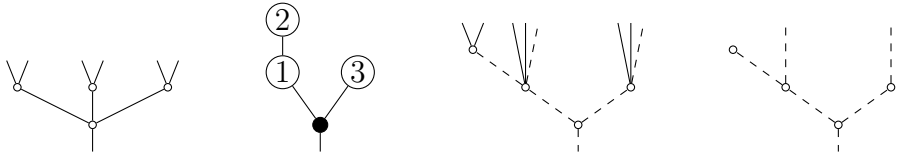
\end{enumerate}

\begin{remark}
Let us informally explain the differentials. The first differential $\d_1$ is simply the sum of the internal differentials on $\cC_1$ and $\cC_2$. The differentials $\d_2$ and $\d_3$ are analogous to the cobar differentials \eqref{eq:cobardifferential} on $\Omega\cC_1$ and $\Omega\cC_2$ respectively. Finally, the differential $\d_4$ expresses the fact that $A\rightarrow \bZ(B)$ is an $\infty$-morphism of $\Omega\cC_1$-algebras and the complicated structure comes from the fact that $\bZ(B)$ is an $\Omega\cC_1$-algebra via the morphism $F\colon \Omega\cC_1\rightarrow \Br_{\cC_2}\{n\}$.
\end{remark}

\begin{lm}
The total differential $\d$ on $\Free(\cP(\cC_1, \cC_2)[-1])$ squares to zero.
\end{lm}
\begin{proof}
The claim in arities $(\cA^{\otimes m}, \cA)$ and $(\cB^{\otimes l}, \cB)$ follows from Lemma \ref{lm:cobardifferential}.

Let us split the differentials on the generators in arities $(\cA^{\otimes -}, \cA)$ and $(\cB^{\otimes -}, \cB)$ as $\d=\d_1+\d_\cA$ and $\d=\d_1+\d_\cB$ respectively.

Given an element $\s^{-n} X\otimes Y$ for $X\in \cC_1(m)$ and $Y\in \cC_2^{\cu}(l)$ the expression $\d^2 (\s^{-n} X\otimes Y)$ splits into the following combinations:
\begin{enumerate}
\item $\d_1^2(\s^{-n} X\otimes Y)$,

\item $(\d_1 \d_2 + \d_2 \d_1)(\s^{-n} X\otimes Y)$,

\item $(\d_1 \d_3 + \d_3 \d_1)(\s^{-n} X\otimes Y)$,

\item $(\d_1 \d_4 + \d_4 \d_1)(\s^{-n} X\otimes Y)$,

\item $(\d_2^2 + \d_\cA\d_2)(\s^{-n} X\otimes Y)$,

\item $(\d_3^2 + \d_\cB\d_3)(\s^{-n} X\otimes Y)$,

\item $(\d_2\d_3 + \d_3\d_2)(\s^{-n} X\otimes Y)$,

\item $(\d_2\d_4 + \d_4\d_2)(\s^{-n} X\otimes Y)$,

\item $(\d_3\d_4 + \d_4\d_3)(\s^{-n} X\otimes Y)$,

\item $(\d_4^2 + \d_\cB \d_4)(\s^{-n} X\otimes Y)$.
\end{enumerate}

We claim that each of these is zero. It is obvious for terms of type $(1)$. Terms of type $(2)$ and $(3)$ vanish due to compatibility of the cooperad structure on $\cC_1$ and $\cC_2^{\cu}$ respectively with the differentials. The vanishing of terms of type $(5)$ and $(6)$ follows as in Lemma \ref{lm:cobardifferential}. The vanishing of the terms of type $(7)$, $(8)$, $(9)$ is obvious as the corresponding modifications of the trees are independent.

Differentials on both $\Omega\cC_1$ and $\Br_{\cC_2}\{n\}$ have a linear and a quadratic component. Therefore, the compatibility of the morphism $F\colon \Omega \cC_1\rightarrow \Br_{\cC_2}\{n\}$ with differentials has two implications. First, the compatibility of the linear parts of the differentials implies the vanishing of terms of type $(4)$. Second, the compatibility of the quadratic parts of the differentials implies the vanishing of terms of type $(10)$.
\end{proof}

\begin{defn}
The \emph{Swiss-cheese operad} $\SC(\cC_1, \cC_2)$ is the colored operad $\Free(\cP(\cC_1, \cC_2)[-1])$ equipped with the above differential.
\end{defn}

We define the $L_\infty$ algebra $\cL(\cC_1, \cC_2; A, B)$ as follows. As a complex,
\[\cL(\cC_1, \cC_2; A, B) = \Conv(\cC_1; A)\oplus \Conv(\cC_2\{n\}; B)\oplus \Hom(\cC_1(A)\otimes \cC_2^{\cu}\{n\}(B), B)[-n-1].\]

The $L_\infty$ operations are given by the following rule:
\begin{itemize}
\item The first two terms have the standard convolution algebra brackets.

\item The first two terms act on the third term by precomposition, i.e. by using the pre-Lie structure on the convolution algebras.

\item The second term acts on the third term by post-composition.

\item Given $R_1,\dots, R_q\in \Conv(\cC_2\{n\}; B)$ and $T_1,\dots, T_r\in \Hom(\cC_1(A)\otimes \cC_2^{cu}\{n\}(B), B)$, their bracket is
\[[R_1, \dots, R_q. T_1, \dots, T_r](X\otimes Y; a_1, \dots, a_m, b_1, \dots, b_l)\]
for $X\in \cC_1(m)$ and $Y\in\cC_2^{\cu}\{n\}(l)$ is given by the sum over pitchforks $\t\in\Isom_\pitchfork(m, r)$ where each term is given as follows. Let $\Delta_\t(X)=X_{(0)}\otimes \dots$ where $X_{(0)}$ is assigned to the root and recall the tree $\t\circ_0 F(\s X_{(0)})$. The value of the bracket is given by the sum over all ways of assigning $T_1, \dots, T_r$ to the white external vertices and $R_1, \dots, R_q$ to the black internal vertices of $\t\circ_0 F(\s X_{(0)})$ and then reading off the composition using the pattern given by the tree.
\end{itemize}

\begin{prop}
The space of morphisms $\Map_{2\Op_k}(\SC(\cC_1, \cC_2), \End_{A, B})$ is equivalent to the space of Maurer--Cartan elements in the $L_\infty$ algebra $\cL(\cC_1, \cC_2; A, B)$.
\label{prop:SCcylinderla}
\end{prop}
The proof of this Proposition is similar to the proof of Proposition \ref{prop:convolutionla}, so we omit it.

Let $B$ be a $\Omega(\cC_2\{n\})$-algebra and consider its center $\bZ(B)$ (see Definition \ref{def:center}), which is a $\Br_{\cC_2}\{n\}$-algebra. Using the morphism
\[F\colon \Omega\cC_1\rightarrow \Br_{\cC_2}\{n\}\]
one defines an $\Omega\cC_1$-algebra structure on $\bZ(B)$. From Proposition \ref{prop:SCcylinderla} and the formulas for the $L_\infty$ brackets we get the following description of $\SC(\cC_1, \cC_2)$-algebras.

\begin{cor}
An algebra over the colored operad $\SC(\cC_1, \cC_2)$ is an $\Omega\cC_1$-algebra $A$, an $\Omega(\cC_2\{n\})$-algebra $B$ and an $\infty$-morphism of $\Omega\cC_1$-algebras $A\rightarrow \bZ(B)$.
\label{cor:SCalgebra}
\end{cor}

Let us now define a small model of the colored operad $\SC(\cC_1, \cC_2)$. Suppose $\Omega\cC_1\rightarrow \cO_1$ and $\Omega(\cC_2\{n\})\rightarrow \cO_2$ are quasi-isomorphisms of operads. An algebra over $\SC(\cC_1, \cC_2)$ is a triple consisting of a homotopy $\cO_1$-algebra $A$, a homotopy $\cO_2$-algebra $B$ and an $\infty$-morphism of homotopy $\cO_1$-algebras $A\rightarrow \bZ(B)$. Define the colored operad $\SC(\cO_1, \cO_2)$ to be the quotient of $\SC(\cC_1, \cC_2)$ whose algebras are triples consisting of a \emph{strict} $\cO_1$-algebra $A$, a \emph{strict} $\cO_2$-algebra $B$ and a \emph{strict} morphism of homotopy $\cO_1$-algebras $A\rightarrow \bZ(B)$.

For the following statement we assume that $\cC_i$ and $\cO_i$ admit an increasing exhaustive filtration $\{F_n\}_{n\geq 0}$ satisfying the following properties:
\begin{enumerate}
\item $F_0 \cC_i=\bu$ and $F_0 \cO_i=\bu$.

\item The morphisms $\Omega\cC_i\rightarrow \cO_i$ are compatible with filtrations.
\end{enumerate}

\begin{prop}
The projection $\SC(\cC_1, \cC_2)\rightarrow \SC(\cO_1, \cO_2)$ is a quasi-isomorphism.
\label{prop:smallSC}
\end{prop}
\begin{proof}
Consider an intermediate operad $\SC'(\cO_1, \cO_2)$ whose algebras are triples consisting of a strict $\cO_1$-algebra $A$, a strict $\cO_2$-algebra $B$ and an \emph{$\infty$-morphism} of homotopy $\cO_1$-algebras $A\rightarrow \bZ(B)$. Thus, we have a sequence of projections
\[\SC(\cC_1, \cC_2)\longrightarrow \SC'(\cO_1, \cO_2)\longrightarrow \SC(\cO_1, \cO_2).\]
We will prove that each morphism is a quasi-isomorphism.

Define the colored operad $(\cO_1)_{\cA}\oplus (\cO_2)_{\cB}$ whose set of colors is $\{\cA, \cB\}$ and whose algebras are given by pairs of an $\cO_1$-algebra and an $\cO_2$-algebra and similarly for $(\Omega\cC_1)_{\cA}\oplus (\Omega(\cC_2\{n\}))_{\cB}$. By construction we have a pushout of operads.
\[
\xymatrix{
(\Omega\cC_1)_{\cA}\oplus (\Omega(\cC_2\{n\}))_{\cB} \ar[d] \ar^-{\sim}[r] & (\cO_1)_{\cA}\oplus (\cO_2)_{\cB} \ar[d] \\
\SC(\cC_1, \cC_2) \ar[r] & \SC'(\cO_1, \cO_2)
}
\]

But the left vertical morphism is a cofibration and the top morphism is a quasi-isomorphism, hence the bottom morphism $\SC(\cC_1, \cC_2)\rightarrow \SC'(\cO_1, \cO_2)$ is a quasi-isomorphism.

We can identify the projection $\SC'(\cO_1, \cO_2)$ in arities $(\cA^{\otimes -}\otimes \cB^0, \cB)$ with the morphism of symmetric sequences $\cC_1\circ_{\d} \cO_1\rightarrow \bu$, where the differential $\d$ comes from the morphism $\Omega\cC_1\rightarrow \cO_1$. But since the latter morphism is a quasi-isomorphism, by \cite[Theorem 6.6.2]{LV} the former morphism is a quasi-isomorphism as well. The claim in arities $(\cA^{\otimes -}\otimes \cB^l, \cB)$ is proved similarly.
\end{proof}

\subsection{Relative Poisson algebras}
\label{sect:relPnalgebras}

In this section we define the main operad to be used in this paper.

\begin{defn}
Let $B$ be a $\bP_n$-algebra. Its \emph{strict Poisson center} is defined to be the $\bP_{n+1}$-algebra
\[\rZ(B) = \Hom_{\mod_B}(\Sym_B(\Omega^1_B[n]), B).\]
\end{defn}
The Lie bracket on $\rZ(B)$ is given by the Schouten bracket of polyvector fields, we refer to \cite[Section 1.1]{Sa1} for explicit formulas.

\begin{defn}
A unital $\bP_{[n+1, n]}$-algebra consists of the following triple:
\begin{itemize}
\item A unital $\bP_{n+1}$-algebra $A$,
\item A unital $\bP_n$-algebra $B$,
\item A morphism of unital $\bP_{n+1}$-algebras $f\colon A\rightarrow \rZ(B)$.
\end{itemize}
\end{defn}
We denote by $\bP_{[n+1, n]}$ the colored operad controlling such algebras, see \cite[Section 1.3]{Sa1} for explicit relations in the operad. Similarly, we denote by $\bP_{[n+1, n]}^{\nu}$ the non-unital version of this operad.

Given a $\bP_{[n+1, n]}$-algebra $(A, B, f)$, the morphism
\[A\longrightarrow \rZ(B)\longrightarrow B\]
is strictly compatible with the multiplications, so it defines a forgetful functor
\[\alg_{\bP_{[n+1, n]}}\longrightarrow \Arr(\alg_{\Comm}).\]

Our goal now is to define a cofibrant resolution of the operad $\bP_{[n+1, n]}$. The cooperad $\coP_n^\theta$ of curved non-unital $\bP_n$-coalgebras has a Hopf counital structure $\coP_n^{\theta, \cu}$ given by the cooperad of curved counital $\bP_n$-coalgebras. Calaque and Willwacher \cite{CW} define a morphism of operads
\begin{equation}
\Omega(\coP_{n+1}^\theta\{1\})\rightarrow \Br_{\coP_n^\theta}
\label{eq:CWmorphism}
\end{equation}
on the generators by the following rule:
\begin{itemize}
\item The generators \[x\in\coLie^\theta\{1-n\}(k)\subset \coP_{n+1}^\theta\{1\}(k)\] are sent to the tree drawn in Figure \ref{fig:centerhomotopymultiplication} with the root labeled by the element \[x\in\coLie^\theta\{1-n\}(k)\subset \coP_n^\theta(k).\]

\item The generator \[x_1\wedge x_2\in\coComm\{1\}(2)\subset \coP^\theta_{n+1}\{1\}(2)\] is sent to the linear combination of trees shown in Figure \ref{fig:centerhomotopybracket}.

\item Given $y\in \coLie^\theta\{1-n\}(k-1)\subset \coP_{n+1}^\theta\{1\}(k-1)$ for $k>2$, we denote by $x\wedge y$ its image under $\coP_{n+1}^\theta\{1\}(k-1)\rightarrow \coP_{n+1}^\theta\{1\}(k)$.

The generators \[x\wedge y\in\coP_{n+1}^\theta\{1\}(k)\] are sent to the tree shown in Figure \ref{fig:centerhomotopyleibniz} with the root labeled by the element
\[y\in\coLie^\theta\{1-n\}(k-1)\subset \coP_n^\theta(k-1).\]

\item The rest of the generators are sent to zero.
\end{itemize}

\begin{figure}
\begin{minipage}{.3\textwidth}
\centering
\begin{tikzpicture}
\node[b] (v0) at (0, 0) {};
\node (v) at (0, -1) {};
\node[w] (v1) at (-1, 1) {$1$};
\node at (0, 1) {\dots};
\node[w] (v2) at (1, 1) {$k$};
\draw (v0) edge (v);
\draw (v0) edge (v1);
\draw (v0) edge (v2);
\end{tikzpicture}
\caption{Image of $\underline{x_1\dots x_k}$.}
\label{fig:centerhomotopymultiplication}
\end{minipage}
\begin{minipage}{.3\textwidth}
\centering
\begin{tikzpicture}
\node[w] (v1) at (-1, 0) {$1$};
\node[w] (v2) at (-1, 1) {$2$};
\node (root) at (-1, -1) {};
\draw (v1) edge (v2);
\draw (v1) edge (root);
\node at (0, 0) {$-(-1)^n$};
\node[w] (v2) at (1, 0) {$2$};
\node[w] (v1) at (1, 1) {$1$};
\node (root) at (1, -1) {};
\draw (v2) edge (v1);
\draw (v2) edge (root);
\end{tikzpicture}
\caption{Image of $x_1\wedge x_2$.}
\label{fig:centerhomotopybracket}
\end{minipage}
\begin{minipage}{.3\textwidth}
\centering
\begin{tikzpicture}
\node[w] (v0) at (0, 0) {$1$};
\node (v) at (0, -1) {};
\node[w] (v1) at (-1, 1) {$2$};
\node at (0, 1) {\dots};
\node[w] (v2) at (1, 1) {$k$};
\draw (v0) edge (v);
\draw (v0) edge (v1);
\draw (v0) edge (v2);
\end{tikzpicture}
\caption{Image of $x_1\wedge\underline{x_2\dots x_k}$}
\label{fig:centerhomotopyleibniz}
\end{minipage}
\end{figure}

Note that the composite
\[\Omega(\coComm\{n+1\})\subset \Omega(\coP_{n+1}\{n+1\})\rightarrow \Br_{\coP_n}\{n\}\]
gives a strict Lie structure and it coincides with the morphism $\Lie\rightarrow \preLie$ as easily seen from Figure \ref{fig:centerhomotopybracket}. We define
\[\widetilde{\bP}_{[n+1, n]} = \SC(\coP_{n+1}^\theta\{n+1\}, \coP_n^\theta),\]
so by Corollary \ref{cor:SCalgebra} a $\widetilde{\bP}_{[n+1, n]}$-algebra is a homotopy unital $\bP_{n+1}$-algebra $A$, a homotopy unital $\bP_n$-algebra $B$ and an $\infty$-morphism of homotopy unital $\bP_{n+1}$-algebras $A\rightarrow \bZ(B)$.

Similarly, one has a morphism of operads
\[\Omega(\coP_{n+1}\{1\})\rightarrow \Br_{\coP_n}\]
and thus we can define the non-unital version of the operad $\widetilde{\bP}^{\nu}_{[n+1, n]}$:
\[\widetilde{\bP}^{\nu}_{[n+1, n]} = \SC(\coP_{n+1}\{n+1\}, \coP_n).\]

\begin{prop}
The natural morphism of colored operads
\[\widetilde{\bP}^{\nu}_{[n+1, n]}\longrightarrow \bP^{\nu}_{[n+1, n]}\]
is a quasi-isomorphism.
\label{prop:relPnresolution}
\end{prop}
\begin{proof}
Let $\bP^w_{[n+1, n]}$ be the colored operad obtained as a quotient of $\widetilde{\bP}^{\nu}_{[n+1, n]}$ whose algebras are strict non-unital $\bP_{n+1}$-algebras $A$, strict non-unital $\bP_n$-algebras $B$ and a strict morphism of homotopy $\bP_{n+1}$-algebras $A\rightarrow \bZ(B)$. We have morphisms of colored operads
\[\widetilde{\bP}^{\nu}_{[n+1, n]}\longrightarrow \bP^w_{[n+1, n]}\longrightarrow \bP^{\nu}_{[n+1, n]},\]
where the first morphism is a quasi-isomorphism by Proposition \ref{prop:smallSC} where we note that $\bP^w_{[n+1, n]} = \SC(\bP^{\nu}_{n+1}, \bP^{\nu}_n)$. Therefore, it is enough to show that the second morphism is a quasi-isomorphism. For this it is enough to show that the morphism in arities $(\cA^{\otimes m}\otimes \cB^{\otimes -}, \cB)$ is a quasi-isomorphism. The following argument is similar to the proof of \cite[Proposition 3.4]{Sa2}.

The relevant morphism is an isomorphism if $m=0$. We will give the proof for $m=1$ since the case of higher $m$ is similar.

If $B$ is a $\bP_n$-algebra and $M$ a $\bP_n$-module over $B$, we can consider its Poisson homology $\C^{\bP_n}_\bullet(B, M)$ which as a graded vector space is isomorphic to
\[\coP_n(M[n])\otimes_k B\cong \Sym_{\geq 1}(\coLie(M[1])[n-1])\otimes_k B.\]
We also have the \emph{canonical} chain complex $\C^{can}_\bullet(B, M)$ which as a graded vector space is isomorphic to $\Sym_{\geq 1}(\Omega^1_M[n])\otimes_k B$. We refer to \cite[Section 1.3, Section 1.4.2]{Fre} for explicit formulas for the differentials on these complexes.

Consider an arbitrary complex $V$. We can identify the colored operad $\bP^w_{[n+1, n]}$ in arity $(\cA^{\otimes 1}\otimes \cB^{\otimes l}, \cB)$ with the coefficient of $V^{\otimes l}$ in $\C^{\bP_n}_\bullet(\bP_n(V), \bP_n(V))$. Similarly, we can identify the colored operad $\bP^{\nu}_{[n+1, n]}$ in the same arity with the coefficient of $V^{\otimes l}$ in $\C^{can}_\bullet(\bP_n(V), \bP_n(V))$. But since $\bP_n(V)$ is free, the natural projection
\[\C^{\bP_n}_\bullet(\bP_n(V), \bP_n(V))\rightarrow \C^{can}_\bullet(\bP_n(V), \bP_n(V))\]
is a quasi-isomorphism which proves the claim.
\end{proof}

One proves similarly that the morphism of colored operads $\widetilde{\bP}_{[n+1, n]}\rightarrow \bP_{[n+1, n]}$ is a quasi-isomorphism.

Let us now explain how to construct graded versions of these operads. Recall that throughout the paper we consider the grading on $\bP_n$ such that the bracket has weight $-1$ and the multiplication has weight $0$. In this section we consider a different convention where the bracket has weight $0$, multiplication weight $1$ and the unit weight $-1$. In particular, $A$ is a graded $\bP_n$-algebra with respect to the current convention iff $A\otimes k(-1)$ is a graded $\bP_n$-algebra in the original convention.

Define the grading on the Hopf cooperad $\coP_n^{\theta, \cu}$ to be such that the comultiplication has weight $0$ and the cobracket has weight $1$. It is compatible with the Hopf structure making it into a graded Hopf cooperad. Also observe that the morphisms
\[\Omega(\coP_n^{\theta}\{n\})\longrightarrow \bP_n\]
and \eqref{eq:CWmorphism} are compatible with gradings. Thus, if $A$ is a graded $\bP_n$-algebra, its Poisson center $\bZ(A)$ becomes a graded homotopy $\bP_{n+1}$-algebra.

\subsection{From relative Poisson algebras to Poisson algebras}
\label{sect:relpoistopois}

Suppose that one has a $\bP_{[n+1, n]}$-algebra $(A, B, f)$ in $M$. In this section we show how to produce a $\bP_{n+1}$-structure on the homotopy fiber of the underlying map of commutative algebras $A\rightarrow B$.

Recall that the strict Poisson center of a $\bP_n$-algebra $B$ is defined to be the algebra of polyvectors
\[\rZ(B) = \Hom_B(\Sym_B(\Omega^1_B[n]), B)\]
with the differential twisted by the Maurer--Cartan element $[\pi_B, -]$ defining the Poisson bracket.

\begin{defn}
Let $B$ be a $\bP_n$-algebra in $M$. Its \emph{strict deformation complex} is defined to be the algebra of polyvectors
\[\Def(B)= \Hom_B(\Sym_B^{\geq 1}(\Omega^1_B[n]), B)\]
with the differential twisted by $[\pi_B, -]$.
\end{defn}
As for the center, the deformation complex $\Def(B)[-n]$ is a $\bP_{n+1}$-algebra, albeit non-unital.

We have a fiber sequence
\[\Def(B)[-n]\rightarrow \rZ(B)\rightarrow B\]
in $M$. Rotating it, we obtain a homotopy fiber sequence
\begin{equation}
B[-1]\rightarrow \Def(B)[-n]\rightarrow \rZ(B),
\label{eq:centersequence}
\end{equation}
where the morphism $B[n-1]\rightarrow \Def(B)$ is given by $b\mapsto [\pi_B, b]$. Note that $B[n-1]$ is a Lie algebra with respect to the Poisson bracket.

\begin{prop}
Let $B$ be a $\bP_n$-algebra in $M$. Then the morphism
\[B[n-1]\rightarrow \Def(B)\]
given by $b\mapsto [\pi_B, b]$ is a morphism of Lie algebras.
\label{prop:centersequencelie}
\end{prop}
\begin{proof}
Indeed, compatibility with the brackets is equivalent to the equation
\[\{\{b_1, b_2\}, b_3\} = \{b_1, \{b_2, b_3\}\} - (-1)^{|b_1||b_2|} \{b_2, \{b_1, b_3\}\}\]
for $b_i\in B[n-1]$ which is exactly the Jacobi identity in the Lie algebra $B[n-1]$.
\end{proof}

Since the morphism $\Def(B)[-n]\rightarrow \rZ(B)$ is a morphism of non-unital $\bP_{n+1}$-algebras, the sequence \eqref{eq:centersequence} can be upgraded to a fiber sequence in $\balg_{\bP_{n+1}^{\nu}}(\cM)$. In particular, $B[-1]$ carries a homotopy non-unital $\bP_{n+1}$-structure such that the underlying Lie structure by Proposition \ref{prop:centersequencelie} coincides with the one on $B$.

\begin{remark}
The fiber sequence of non-unital $\bP_{n+1}$-algebras \eqref{eq:centersequence} is a $\bP_n$-analog of the sequence
\[B[-1]\rightarrow \Def(B)[-n]\rightarrow \mathrm{HH}^\bullet_{\bE_n}(B)\]
of non-unital $\bE_{n+1}$-algebras constructed by Francis in \cite[Theorem 4.25]{Fr} (see also \cite[Section 5.3.2]{Lu}) when $B$ is an $\bE_n$-algebra.
\end{remark}

Next, suppose $(A, B, f)$ is a $\bP_{[n+1, n]}$-algebra in $M$. Then we can construct a commutative diagram of non-unital $\bP_{n+1}$-algebras
\[
\xymatrix{
B[-1]\ar@{=}[d] \ar[r]&\U(A,B) \ar[r] \ar[d] & A \ar[d]\\
B[-1] \ar[r] & \Def(B)[-n] \ar[r] & \rZ(B)
}
\]
where both rows are homotopy fiber sequences and the square on the right is Cartesian. This defines $\U(A, B)$, a non-unital $\bP_{n+1}$-algebra, which as a Lie algebra fits into a fiber sequence
\[B[n-1]\rightarrow \U(A, B)[n]\rightarrow A[n].\]

Proceeding as in \cite[Definition 1.4.15]{CPTVV} it defines a forgetful functor
\[\U\colon \balg_{\bP_{[n+1, n]}}(\cM)\longrightarrow \balg_{\bP_{n+1}^{\nu}}(\cM).\]

Let us now give explicit formulas for the Lie brackets on $\U(A, B)$. From the $\bP_{n+1}$-morphism $A\rightarrow \rZ(B)$ we obtain a morphism of Lie algebras
\[A[n]\rightarrow \rZ(B)[n]\rightarrow \Hom(\Sym(B[n]), B[n]),\]
where the object on the right is equipped with the convolution Lie bracket. Therefore, by results of \cite[Section 3]{La} we obtain an $L_\infty$ structure on $A[n]\oplus B[n-1]$. The brackets are as follows:
\begin{itemize}
\item The bracket on $A[n]$ is the Poisson bracket on $A$,

\item The bracket on $B[n-1]$ is the Poisson bracket on $B$,

\item The mixed brackets $A[n]\otimes (B[n-1])^{\otimes k}\rightarrow B[n-k]$ are given by the components $f_k\colon A\rightarrow \Hom(\Sym^k(B[n]), B)$ of the morphism above.
\end{itemize}

\subsection{Mixed structures on relative Poisson algebras}
\label{sect:mixedpoisson}

Let $A$ be a graded $\bP_{n+1}$-algebra in $\cM$. Moreover, suppose we have a morphism
\[a\colon k(2)[-1]\longrightarrow A[n]\]
of graded Lie algebras. The bracket defines a morphism $A[n]\rightarrow \Der(A, A)\otimes k(1)$ of graded Lie algebras and hence we obtain a morphism
\[k(2)[-1]\longrightarrow \Der(A, A)\otimes k(1)\]
of graded Lie algebras, i.e. the graded commutative algebra $A$ is enhanced to a graded mixed algebra. Let $A_0$ be the weight 0 component of $A$. By the universal property of the de Rham algebra (see \cite[Proposition 1.3.8]{CPTVV}) we obtain a morphism
\[\bDR^{int}(A_0)\longrightarrow A\]
of graded mixed commutative algebras.

More explicitly, given an element $a\in A$ we construct a graded mixed structure on $A$ given by $\epsilon_A = [a, -]$. The morphism $\bDR^{int}(A_0)\rightarrow A$ is given by
\[x \ddr y_1 \ddr y_2 \dots \ddr y_n\mapsto x [a, y_1][a,y_2] \dots [a, y_n].\]

We proceed to a relative version of this construction. Fix a graded $\bP_{[n+1, n]}$-algebra $(A, B, f)$ in $M$ and suppose
\[(a, b)\colon k(2)[-1]\longrightarrow \U(A, B)[n]\]
is a $\infty$-morphism of graded $L_\infty$-algebras. By Proposition \ref{prop:MCgradeddglie} it is given by a collection of elements $a=a_2+\dots$ and $b=b_2+\dots$ satisfying the Maurer--Cartan equation, where the subscript denotes the weight. Twisting the $L_\infty$ brackets on $\U(A, B)[n]$ by $(a, b)$ as in Definition \ref{def:MCtwist}, we obtain mixed structures on $A$ and $B$. Using the explicit description of the $L_\infty$ brackets on $\U(A, B)[n]$ given in Section \ref{sect:relpoistopois}, we see that the mixed structure $\epsilon_A$ on $A$ is given by $\epsilon_A = [a, -]$ while the mixed structure $\epsilon_B$ on $B$ is given by
\[\epsilon_B = [b, -] + \sum_{k=1}^\infty \frac{1}{(k-1)!}f_k(a; b, \dots, b, -).\]

It is clear from these formulas that the mixed structures $\epsilon_A$ and $\epsilon_B$ are derivations of the commutative multiplication, thus both $A$ and $B$ become weak graded mixed commutative algebras. Moreover, the morphism $f_0\colon A\rightarrow B$ extends to an $\infty$-morphism of weak graded mixed commutative algebras whose components are given by $\frac{1}{k!}f_k(-; b, \dots, b)$. The corresponding relation between the mixed structures on $A$ and $B$ simply follows by observing that twisting the differential on an $L_\infty$-algebra produces a mixed structure.

In this way we construct an $\infty$-morphism of $L_\infty$-algebras
\[\U(A, B)[n]\rightarrow \Der(A\rightarrow B, A\rightarrow B)\otimes k(1),\]
where we consider $A\rightarrow B$ as a commutative algebra in $\Arr(M)$. Moreover, this can be enhanced to a natural transformation of functors
\[\balg_{\bP_{[n+1, n]}}^{gr}(\cM)^{\sim}\rightarrow \balg_{\Lie}^{gr}.\]

Considering again $A_0\rightarrow B_0$ as a commutative algebra in $\Arr(\cM)$ we can identify
\[\bDR^{int}(A_0\rightarrow B_0)\cong (\bDR^{int}(A_0)\rightarrow \bDR^{int}(B_0)).\]
Therefore, given a graded $\bP_{[n+1, n]}$-algebra $(A, B, f)$ we obtain a diagram
\[
\xymatrix{
\bDR^{int}(A_0) \ar[d] \ar[r] & \bDR^{int}(B_0) \ar[d] \\
A \ar[r] & B
}
\]
of graded mixed commutative algebras in $\cM$, where $A$ and $B$ are endowed with the graded mixed structure constructed above.

\section{Coisotropic structures on affine derived schemes}

In Sections \ref{sect:operadicres} and \ref{sect:braceconstruction} we have established the necessary facts about convolution and relative convolution algebras. In this section we use those notions to define polyvector and relative polyvector fields. We also define the notion of Poisson and coisotropic structures and compute these spaces explicitly in terms of Maurer--Cartan elements in the Lie algebras of polyvectors and relative polyvectors respectively.

\subsection{Poisson structures and polyvectors}
\label{sect:polyvectors}

Let $A$ be a commutative algebra in $M$ which we view as a graded $\bP_{n+1}$-algebra with the trivial bracket. Recall from Definition \ref{def:center} its center $\bZ(A)$ which is a graded $\Br_{\coP_{n+1}^\theta}$-algebra and hence a graded homotopy $\bP_{n+2}$-algebra.

\begin{defn}
Let $A$ be a commutative algebra in $M$. Its algebra of \emph{$n$-shifted polyvectors} $\bPol(A, n)$ is the graded homotopy $\bP_{n+2}$-algebra $\bZ(A)$.
\end{defn}

By considering the internal Hom in $M$, we can upgrade $\bPol(A, n)$ to a graded homotopy $\bP_{n+2}$-algebra in $M$ that we denote by $\bPol^{int}(A, n)$ and call the algebra of \emph{internal $n$-shifted polyvectors}.

Let $\CAlg^{fet}(M)\subset \CAlg(M)$ be the wide subcategory of commutative algebras in $M$ where we only consider morphisms which are formally \'{e}tale, i.e. morphisms $A\rightarrow B$ such that the pullback morphism $\Omega^1_A\otimes_A B\rightarrow \Omega^1_B$ is a quasi-isomorphism. Denote by $\bCAlg^{fet}(\cM)$ its localization. Then as in \cite[Definition 1.4.15]{CPTVV} one can upgrade $\bPol^{int}(-, n)$ to a functor of $\infty$-categories
\[\bPol^{int}(-, n)\colon \bCAlg^{fet}(\cM)\longrightarrow \balg_{\bP_{n+2}}(\cM^{gr}).\]

The complex
\[\Pol(A, n) = \Hom_A(\Sym_A(\Omega^1_A[n+1]), A)\]
carries a natural graded $\bP_{n+2}$-algebra structure where $\Omega^1_A$ has weight $-1$ and where the Lie bracket is given by the Schouten bracket (see \cite[Section 1.1]{Sa1} for explicit formulas). The following statement explains the term ``polyvectors'':
\begin{prop}
Let $A$ be a bifibrant commutative algebra in $M$. Then one has an equivalence of graded $\bP_{n+2}$-algebras
\[\bPol(A, n)\cong \Pol(A, n).\]
\label{prop:strictpolyvectors}
\end{prop}
\begin{proof}
By definition we have
\[\bPol(A, n)\cong \Hom_A(\Sym_A(\Harr_\bullet(A, A)[n+1]) A).\]

The morphism $\Harr_\bullet(A, A)\rightarrow \Omega^1_A$ induces a morphism
\[\Pol(A, n)\longrightarrow \bPol(A, n)\]
which is strictly compatible with the homotopy $\bP_{n+2}$-structures by \cite[Theorem 1]{CW}.

Since $A$ is a cofibrant commutative algebra, it is also cofibrant in $M$, so by Proposition \ref{prop:harrisoncotangent} the morphism $\Harr_\bullet(A, A)\rightarrow \Omega^1_A$ is a weak equivalence between cofibrant $A$-modules and hence so is
\[\Sym_A(\Harr_\bullet(A, A)[n+1])\longrightarrow \Sym_A(\Omega^1_A[n+1]).\]

Since $A$ is fibrant, the functor $\Hom_A(-, A)$ preserves weak equivalences.
\end{proof}

\begin{cor}
For any commutative algebra $A$ in $\cM$ we have an equivalence of graded objects
\[\bPol(A, n)\cong \Hom_A(\Sym_A(\bL_A[n+1]), A).\]

In particular, if $\bL_A$ is perfect, we get an equivalence
\[\bPol(A, n)\cong \Sym_A(\bT_A[-n-1]).\]
\end{cor}

\begin{defn}
Let $A$ be a commutative algebra in $\cM$, The \emph{space of $n$-shifted Poisson structures} $\Pois(A, n)$ is given by the fiber of the forgetful functor
\[\balg_{\bP_{n+1}}(\cM)^{\sim}\rightarrow \bCAlg(\cM)^{\sim}\]
at $A\in\bCAlg(\cM)$.
\end{defn}

One has an explicit way to compute the space of shifted Poisson structures in terms of the algebra of polyvectors. The following theorem is a version of \cite[Theorem 1.4.9]{CPTVV} and \cite[Theorem 3.2]{Me}:
\begin{thm}
Let $A$ be a commutative algebra in $\cM$. One has an equivalence of spaces
\[\Pois(A, n)\cong \Map_{\balg_{\Lie}^{gr}}(k(2)[-1], \bPol(A, n)[n+1]).\]
\label{thm:poissonpolyvectors}
\end{thm}
\begin{proof}
Assume $A$ is a bifibrant commutative algebra in $M$. By \cite{Rez} (see also \cite{Ya2}) we can identify the space $\Pois(A, n)$ with the homotopy fiber of
\[\Map_{\Op_k}(\bP_{n+1}, \End_A)\rightarrow \Map_{\Op_k}(\Comm, \End_A)\]
at the given commutative structure on $A$.

We have a commutative diagram of spaces
\[
\xymatrix{
\Map_{\Op_k}(\bP_{n+1}, \End_A) \ar[r] \ar^{\sim}[d] & \Map_{\Op_k}(\Comm, \End_A) \ar^{\sim}[d] \\
\underline{\MC}(\Conv(\coP_{n+1}^{\theta}\{n+1\}; A)) \ar[r] & \underline{\MC}(\Conv(\coLie^\theta\{1\}; A)
}
\]
with the vertical weak equivalences given by Proposition \ref{prop:convolutionla}.

By Lemma \ref{lm:MCfiber} we can identify the fiber of the bottom map with the space of Maurer--Cartan elements in the dg Lie algebra
\[\Hom_k(\Sym^{\geq 2}(\coLie^{\theta}(A[1])[n]), A)\cong \Hom_A(\Sym^{\geq 2}_A(\Harr_\bullet(A, A)[n+1]), A)\]
which can be identified with the completion of $\bPol(A, n)$ in weights $2$ and above. The claim therefore follows from Proposition \ref{prop:MCgradeddglie}.
\end{proof}

\begin{remark}
The main difference with the formulation we give and the one in \cite{Me} and \cite{CPTVV} is that we consider the non-strict version of polyvectors. To obtain the same result, one should couple Theorem \ref{thm:poissonpolyvectors} with Proposition \ref{prop:strictpolyvectors}.
\end{remark}

If $A$ is a $\bP_{n+1}$-algebra, we define the opposite algebra to have the same multiplication and the bracket $\{a, b\}^{\opp} = -\{a, b\}$. This defines a morphism of spaces
\[\opp\colon \Pois(A, n)\longrightarrow \Pois(A, n).\]
It is easy to see that under the equivalence given by Theorem \ref{thm:poissonpolyvectors} it corresponds to the involution on $\bPol(A, n)$ given by multiplication by $(-1)^{k+1}$ in weight $k$.

\subsection{Relative polyvectors}
\label{sect:relativepoly}

Suppose $f\colon A\rightarrow B$ is a morphism of commutative algebras in $M$. We regard $A$ as a graded $\bP_{n+1}$-algebra with the trivial bracket and $B$ as a graded $\bP_n$-algebra with the trivial bracket. In particular, the composite
\[A\stackrel{f}\longrightarrow B\longrightarrow \bPol^{int}(B, n-1)\]
is strictly compatible with the $\bP_{n+1}$-structures and we denote it by $\tilde{f}$.

Recall from Section \ref{sect:relativebrace} the complex
\[\Conv^0(\coP_{n+1}^{\theta,\cu}\{n+1\}; A, \bPol^{int}(B, n-1)).\]
The morphism $\tilde{f}$ defines a Maurer--Cartan element in the $L_\infty$ algebra \[\cL^0(\coP_{n+1}^{\theta,\cu}\{n+1\}; A, \bPol^{int}(B, n-1))\] and hence we can consider the twisted convolution algebra
\[\Conv^0_{\tilde{f}}(\coP_{n+1}^{\theta,\cu}\{n+1\}; A, \bPol^{int}(B, n-1)).\]

As shown in the same section, it carries a natural structure of a graded homotopy $\bP_{n+1}$-algebra which we denote by $\bPol(B/A, n-1)$. Similarly, we can define its internal version $\bPol^{int}(B/A, n-1)$ which is a graded homotopy $\bP_{n+1}$-algebra in $M$.

\begin{prop}
Let $f\colon A\rightarrow B$ be a cofibrant diagram of commutative algebras in $M$, where $B$ is also fibrant. Then one has an equivalence of graded objects
\[\bPol^{int}(B/A, n-1)\cong \Hom_B(\Sym_B(\Omega^1_{B/A}[n]), B).\]
\label{prop:strictrelpolyvectors1}
\end{prop}
\begin{proof}
Since $B$ is bifibrant as a commutative algebra, it is enough to prove that the projection
\[\Sym(\coLie(A[1])[n])\otimes \Sym_B(\Omega^1_B[n])\rightarrow \Sym(\Omega^1_{B/A}[n])\]
induced by the morphism $\Omega^1_B\rightarrow \Omega^1_{B/A}$ induces a weak equivalence after passing to left realizations. This will follow once we prove that
\[\Harr_\bullet(A, A)\otimes_A B \oplus \Omega^1_B\rightarrow \Omega^1_{B/A}\]
induces a weak equivalence of $B$-modules after passing to left realizations. Here the grading is inherited from the Harrison complex on the left-hand side and given by putting the K\"{a}hler differentials in weight $0$. The mixed structure on the right-hand side is trivial. The mixed structure on the left-hand side is a sum of two terms:
\begin{enumerate}
\item The first term is the usual Harrison differential.
\item The second term is given by the composite
\[A\otimes B\stackrel{\ddr\otimes\id}\rightarrow \Omega^1_A\otimes_A B\rightarrow \Omega^1_B,\]
where $A\otimes B$ is the weight $-1$ part of $\Harr_\bullet(A, A)\otimes_A B$.
\end{enumerate}

Since $A$ is cofibrant, by Proposition \ref{prop:harrisoncotangent} it is enough to prove that
\[\Omega^1_A\otimes_A B[1](-1)\oplus \Omega^1_B\rightarrow \Omega^1_{B/A},\]
where the left-hand side is equipped with the mixed structure given by the pullback of differential forms $\Omega^1_A\otimes_A B\rightarrow \Omega^1_B$ induces a weak equivalence after passing to left realizations. But since $A\rightarrow B$ is a cofibrant diagram, the natural sequence of $B$-modules
\[\Omega^1_A\otimes_A B\rightarrow \Omega^1_B\rightarrow \Omega^1_{B/A}\]
is exact and by Proposition \ref{prop:realizationfiber} this finishes the proof.
\end{proof}

Using the morphism $f\colon A\rightarrow B$ one can regard $B$ as a commutative algebra in $\mod_A(M)$. In particular, we can compute internal strict polyvectors of $B$ in $\mod_A(M)$ which we denote by $\Pol^{int}_A(B, n-1)$. It is easy to see that the composite
\[\Pol^{int}_A(B, n-1)\rightarrow \bPol^{int}(B, n-1)\rightarrow \bPol^{int}(B/A, n-1)\]
is strictly compatible with the homotopy $\bP_{n+1}$-structures.

\begin{remark}
Note that the map $\bPol^{int}(B, n-1)\rightarrow \bPol^{int}(B/A, n-1)$ is \emph{not} compatible with differentials and is not model-independent. However, it is easy to see that the composite map $\Pol^{int}_A(B, n-1)\rightarrow \bPol^{int}(B/A, n-1)$ is compatible with the differentials.
\end{remark}

\begin{cor}
Let $A\rightarrow B$ be a cofibrant diagram of commutative algebras in $M$, where $B$ is also fibrant. Then the morphism
\[\Pol^{int}_A(B, n-1)\longrightarrow \bPol^{int}(B/A, n-1)\]
is a weak equivalence of homotopy $\bP_{n+1}$-algebras in $M$.
\end{cor}
\begin{proof}
It is enough to prove that the morphism is a weak equivalence of graded objects in $M$ which follows from Propositions \ref{prop:strictrelpolyvectors1} and \ref{prop:strictpolyvectors}.
\end{proof}

By Proposition \ref{prop:bracemorphisms} we have an $\infty$-morphism of graded homotopy $\bP_{n+1}$-algebras
\[\bPol^{int}(A, n)\rightarrow \bZ(\bPol^{int}(B/A, n-1))\]
and hence $(\bPol^{int}(A, n), \bPol^{int}(B/A, n-1))$ forms a graded homotopy $\bP_{[n+2, n+1]}$-algebra in $M$. Similarly, $(\bPol(A, n), \bPol(B/A, n-1))$ is a graded homotopy $\bP_{[n+2, n+1]}$-algebra in complexes.

\begin{defn}
Let $f\colon A\rightarrow B$ be a morphism of commutative algebras in $M$. The \emph{algebra of relative $n$-shifted polyvector fields} is the graded homotopy non-unital $\bP_{n+2}$-algebra
\[\bPol(f, n) = \U(\bPol(A, n), \bPol(B/A, n-1)).\]
\end{defn}
As before, we can upgrade it to a functor
\[\bPol(-, n)\colon \Arr(\bCAlg^{fet}(\cM))\rightarrow \balg_{\bP_{[n+2, n+1]}}(\cM^{gr}).\]

Note that as a graded vector space we can identify
\[\bPol(f, n)[n+1]\cong \bPol(A, n)[n+1]\oplus \bPol(B/A, n-1)[n].\]
In particular, we get a morphism of graded vector spaces
\[\bPol(f, n)[n+1]\rightarrow \bPol(B/A, n-1)[n]\rightarrow \bPol(B, n-1)[n].\]

Restricting to weight $0$ we get the morphism $A[n+1]\oplus B[n]\rightarrow B[n]$ which is \emph{not} compatible with the differentials. Nevertheless, in highest weight this phenomenon does not occur.

\begin{prop}
The morphism $\bPol^{\geq 1}(f, n)[n+1]\rightarrow \bPol^{\geq 1}(B, n-1)[n]$ is a morphism of filtered $L_\infty$ algebras.
\end{prop}
\begin{proof}
Recall from Section \ref{sect:relpoistopois} that the $L_\infty$ brackets on $\bPol(A, n)[n+1]\oplus \bPol(B/A, n-1)[n]$ are given by the original brackets on each summand and the mixed $L_\infty$ brackets between $\bPol(A, n)$ and $\bPol(B/A, n-1)$ which land in $\bPol(B/A, n-1)$.

The map $\bPol(B/A,n-1)[n]\rightarrow \bPol(B, n-1)[n]$ is compatible with the Lie structures, so we need to show that the image of the mixed brackets vanishes. By definition they come from the action of the convolution algebra
\[\bPol(A, n)=\Conv^0(\coP_{n+1}^{\theta, \cu}\{n+1\}; A)\]
on the relative convolution algebra
\[\bPol(B/A, n-1) = \Conv^0(\coP_{n+1}^{\theta, \cu}\{n+1\}; A, \bPol(B, n-1))\]
given by the last tree in Figure \ref{fig:bracemorphism}. But elements in $\bPol^{\geq 1}(A, n)$ have at least one $A$ input, so the result in $\Conv^0(\coP_{n+1}^{\theta, \cu}\{n+1\}; A, \bPol(B, n-1))$ will have at least one $A$ input. However, the projection map
\[\Conv^0(\coP_{n+1}^{\theta, \cu}\{n+1\}; A, \bPol(B, n-1))\rightarrow \bPol(B, n-1)\]
annihilates all such elements which proves the claim.
\end{proof}

We get a diagram of Lie algebras
\[
\xymatrix{
& \bPol^{\geq 2}(f, n)[n+1] \ar[dl] \ar[dr] & \\
\bPol^{\geq 2}(B, n-1)[n] && \bPol^{\geq 2}(A, n)[n+1]
}
\]

The Lie algebra $\bPol(f, n)[n+1]$ is quite complicated, so in practice we will use the following strict model. Define the graded non-unital $\bP_{n+2}$-algebra
\[\Pol(f, n) = \ker(\Pol(A, n)\longrightarrow \Pol_A(B, n-1)),\]
where the graded non-unital $\bP_{n+2}$-structure on $\Pol(f, n)$ is induced from the one on $\Pol(A, n)$.

\begin{prop}
Suppose $f\colon A\rightarrow B$ is a bifibrant object of $\Arr(\CAlg(M))$. Moreover, assume that
\[\Pol(A, n)\longrightarrow \Pol_A(B, n-1)\]
is surjective. Then we have a quasi-isomorphism of $L_\infty$ algebras
\[\Pol(f, n)[n+1]\cong \bPol(f, n)[n+1].\]
\label{prop:strictrelpolyvectors}
\end{prop}
\begin{proof}
The Lie algebra $\bPol(f, n)[n+1]$ is given by the complex
\[\bPol(A, n)[n+1]\oplus \bPol(B/A, n-1)[n]\]
with the differential from the first term to the second term given by the morphism \[\bPol(A, n)\longrightarrow \bPol(B/A, n-1)\] and the following brackets:
\begin{itemize}
\item The convolution bracket on $\bPol(A, n)$.

\item The convolution bracket on $\bPol(B/A, n-1)$.

\item The action of $\bPol(A, n)$ on $\bPol(B/A, n-1)$.
\end{itemize}

By Proposition \ref{prop:strictpolyvectors} the morphism $\Pol(A, n)\rightarrow \bPol(A, n)$ is a quasi-isomorphism, so we can replace the above dg Lie algebra with
\[\Pol(A, n)[n+1]\oplus \bPol(B/A, n-1)[n].\]

Moreover, the inclusion $\Pol_A(B, n-1)\rightarrow \bPol(B/A, n-1)$ is a quasi-isomorphism. Extending it to a deformation retract and applying the homotopy transfer theorem \cite[Theorem 10.3.9]{LV} we obtain a quasi-isomorphic $L_\infty$ structure on the complex
\[\widetilde{\Pol}(f, n)[n+1] = \Pol(A, n)[n+1]\oplus \Pol_A(B, n-1)[n]\]
with the differential twisted by the morphism
\[\Pol(A, n)\longrightarrow \Pol_A(B, n-1).\]

Using the explicit formulas for the transferred structure, we see that all the brackets
\[[-, \dots, -]_m\colon (\Pol(A, n)[n+1])^{\otimes m}\rightarrow \Pol_A(B, n-1)[2+n-m]\]
vanish for $m>1$. Therefore, the morphism
\[\Pol(f, n)[n+1]\rightarrow \widetilde{\Pol}(f, n)[n+1]\]
is strictly compatible with the $L_\infty$ brackets and is a quasi-isomorphism by the surjectivity assumption.
\end{proof}

\subsection{Coisotropic structures}
\label{sect:coisotropics}

\begin{defn}
Let $f\colon A\rightarrow B$ be a morphism of commutative algebras in $\cM$. The \emph{space of $n$-shifted coisotropic structures} $\Cois(f, n)$ is defined to be the fiber of the forgetful functor
\[\balg_{\bP_{[n+1, n]}}(\cM)^\sim\longrightarrow \Arr(\bCAlg(\cM))^\sim\]
at $f\in \Arr(\bCAlg(\cM))$.
\end{defn}

One has the following alternative point of view on coisotropic structures following \cite[Section 3.4]{CPTVV}. First of all, one has the following additivity statement for Poisson structures shown independently by Rozenblyum and the second author, see \cite[Theorem 2.22]{Sa2}:
\begin{thm}
One has an equivalence of $\infty$-categories
\[\balg_{\bP_{n+1}}(\cM)\cong \balg(\balg_{\bP_n}(\cM)).\]
\label{thm:poissonadditivity}
\end{thm}

In other words, a $\bP_{n+1}$-algebra is equivalent to an associative algebra object in $\bP_n$-algebras. Let us now denote by $\blmod(\balg_{\bP_n}(\cM))$ the $\infty$-category of pairs $(A, B)$ of an associative algebra object $A$ in $\bP_n$-algebras and a left $A$-module $B$.

The following is \cite[Corollary 3.8]{Sa2}:
\begin{thm}
One has a commutative diagram of $\infty$-categories
\[
\xymatrix{
\balg_{\bP_{[n+1, n]}}(\cM) \ar^{\sim}[rr] \ar[dr] && \blmod(\balg_{\bP_n}(\cM)) \ar[dl] \\
& \Arr(\bCAlg(\cM))
}
\]
\label{thm:relpoissonadditivity}
\end{thm}

Therefore, a coisotropic structure on the morphism $f\colon A\rightarrow B$ consists of a $\bP_{n+1}$-algebra structure on $A$, a $\bP_n$-algebra structure on $B$ and an associative action of $A$ on $B$, all compatible with the original morphism $f$.

Returning to our definition of coisotropic structures, we have forgetful maps
\[
\xymatrix{
& \Cois(f, n) \ar[dl] \ar[dr] & \\
\Pois(B, n-1) && \Pois(A, n).
}
\]

One has the following explicit way to compute this diagram of spaces.

\begin{thm}
Let $f\colon A\rightarrow B$ be a morphism of commutative algebras in $\cM$. Then one has an equivalence of spaces
\[\Cois(f, n)\cong \Map_{\balg_{\Lie}^{gr}}(k(2)[-1], \bPol(f, n)[n+1])\cong \underline{\MC}(\bPol^{\geq 2}(f, n)[n+1])\]
compatible with the diagram of spaces
\[
\xymatrix{
& \Cois(f, n) \ar[dl] \ar[dr] & \\
\Pois(B, n-1) && \Pois(A, n)
}
\]
on the left and the diagram of Lie algebras
\[
\xymatrix{
& \bPol^{\geq 2}(f, n)[n+1] \ar[dl] \ar[dr] & \\
\bPol^{\geq 2}(B, n-1)[n] && \bPol^{\geq 2}(A, n)[n+1]
}
\]
on the right.
\label{thm:coisotropicpolyvectors}
\end{thm}

We will give a proof of this theorem in the next section. Here is one application of the above theorem. Recall that classically the identity morphism from a Poisson manifold to itself is coisotropic. Similarly, we show that in the derived context there is a \emph{unique} coisotropic structure on the identity.

\begin{prop}
Let $A$ be a commutative algebra in $\cM$ and $\id\colon A\rightarrow A$ the identity morphism. Then the forgetful map
\[\Cois(\id, n)\rightarrow \Pois(A, n)\]
is a weak equivalence.
\label{prop:idcoisotropic}
\end{prop}
\begin{proof}
We have a fiber sequence of graded Lie algebras
\[\bPol(A/A, n-1)[n]\longrightarrow \bPol(\id, n)[n+1]\longrightarrow \bPol(A, n)[n+1].\]

But $\bPol(A/A, n-1)\cong k$, hence the morphism $\bPol(\id, n)[n+1]\rightarrow \bPol(A, n)[n+1]$ becomes an equivalence after applying $\Map_{\alg_{\Lie}^{gr}}(k(2)[-1], -)$.
\end{proof}

\begin{defn}
The forgetful map
\[\Pois(A, n)\longrightarrow \Pois(A, n-1)\]
is the composite
\[\Pois(A, n)\stackrel{\sim}\longleftarrow \Cois(\id, n)\longrightarrow \Pois(A, n-1).\]
\end{defn}
This forgetful map is the classical analog of the forgetful functor $\alg_{\bE_{n+1}}\rightarrow \alg_{\bE_n}$.

\subsection{Proof of Theorem \ref{thm:coisotropicpolyvectors}}

Assume $f\colon A\rightarrow B$ is a bifibrant object in $\Arr(M)$ with respect to the projective model structure. Then $\Cois(f, n)$ is equivalent to the homotopy fiber of
\[\Map_{2\Op_k}(\bP_{[n+1, n]}, \End_{A, B})\rightarrow \Map_{\Op_k}(\Comm, \End_{A\rightarrow B})\]
at $f\colon A\rightarrow B$.

By Proposition \ref{prop:convolutionla} we can identify
\[\Map_{\Op_k}(\Comm, \End_{A\rightarrow B})\cong \Conv(\coLie^\theta\{1\}; A\rightarrow B),\]
where we can identify $\Conv(\coLie^\theta\{1\}; A\rightarrow B)$ as a complex with
\[\ker(\Conv(\coLie^\theta\{1\}; A)\oplus \Conv(\coLie^\theta\{1\}; B)\rightarrow \Hom(\coLie^\theta\{1\}(A), B)).\]

Recall from \cite[Section 3.2]{DW2} (a related construction also appears in \cite{FZ}) the cylinder $L_\infty$-algebra \[\cL(\coLie^\theta\{1\}; A, B)=\mathrm{Cyl}(\coLie^\theta\{1\}; A, B)^f\] which fits into a \emph{homotopy} fiber sequence of complexes
\[\cL(\coLie^\theta\{1\}; A, B)\longrightarrow \Conv(\coLie^\theta\{1\}; A)\oplus \Conv(\coLie^\theta\{1\}; B)\longrightarrow \Hom(\coLie^\theta\{1\}(A), B).\]

Since $A\rightarrow B$ is a cofibration in $M$ and $B$ is fibrant, the obvious morphism
\[\Conv(\coLie^\theta\{1\}; A\rightarrow B)\rightarrow \cL(\coLie^\theta\{1\}; A, B)\]
is a quasi-isomorphism. Moreover, it is strictly compatible with the $L_\infty$ brackets.

Combining Propositions \ref{prop:relPnresolution} and \ref{prop:SCcylinderla} we obtain a commutative diagram of spaces
\[
\xymatrix{
\Map_{2\Op_k}(\bP_{[n+1, n]}, \End_{A, B}) \ar[r] \ar^{\sim}[d] & \Map_{2\Op_k}(\Comm, \End_{A\rightarrow B}) \ar^{\sim}[d] \\
\underline{\MC}(\cL(\coP_{n+1}^\theta\{n+1\}, \coP_n^\theta; A, B)) \ar[r] & \underline{\MC}(\cL(\coLie^\theta\{1\}; A, B))
}
\]

By Lemma \ref{lm:MCfiber} we can identify the fiber of the bottom map with the space of Maurer--Cartan elements in the $L_\infty$ algebra
\begin{align*}
&\Hom(\Sym^{\geq 2}(\coLie^\theta(A[1])[n]), A)[n+1]\\
\oplus &\Hom(\Sym^{\geq 2}(\coLie^\theta(B[1])[n-1]), B)[n]\\
\oplus &\Hom(\Sym^{\geq 2}(\coLie^\theta(A[1])[n]) \otimes \Sym(\coLie^\theta(B[1])[n-1]), B)[n]\\
\oplus &\Hom(\Sym^{\geq 1}(\coLie^\theta(A[1])[n]) \otimes \Sym^{\geq 1}(\coLie^\theta(B[1])[n-1]), B)[n].
\end{align*}

We can identify the first term with the completion of $\bPol(A, n)[n+1]$ in weights $2$ and above. The rest of the terms are identified with the completion of $\bPol(B/A, n-1)[n]$ in weights $2$ and above. It is easy to see that this identification is compatible with the Lie bracket on $\bPol(A, n)[n+1]$ and the $L_\infty$ brackets on $\bPol(B/A, n-1)[n]$.

The mixed brackets given by the aciton of $\bPol(A, n)[n+1]$ on $\bPol(B/A, n-1)[n]$ can be identified as follows. Recall from Proposition \ref{prop:bracemorphisms} that the $\infty$-morphism of $\Br_{\coP_{n+1}^\theta}$-algebras
\[\bPol(A, n)[n+1]\longrightarrow \Conv^0_f(\coP_{n+1}^{\theta, \cu}, \bPol(B/A, n-1)[n])\]
has an underlying strict Lie morphism which factors through
\[\bPol(A, n)[n+1]\longrightarrow \End(\bPol(B/A, n-1)[n])\]
and which is given by the precomposition map of the convolution algebra of $A$ on the relative convolution algebra of $B/A$. But this exactly coincides with the mixed bracket in $\cL(\coP_{n+1}^\theta\{n+1\}, \coP_n^\theta; A, B)$ and hence we have identified $\Cois(f, n)$ with the space of Maurer--Cartan elements in $\bPol(f, n)^{\geq 2}$.

The claim therefore follows from Proposition \ref{prop:MCgradeddglie}.

\subsection{Poisson morphisms}

Recall that classically a Poisson morphism $f\colon X\rightarrow Y$ between two Poisson manifolds can be characterized by the property that the graph $X\rightarrow X\times \overline{Y}$ is coisotropic, where $\overline{Y}$ is the same manifold with the opposite Poisson structure. In this section we define a derived notion of Poisson morphisms and show that an analogue of this result holds in the derived setting as well.

\begin{defn}
Suppose $A$ and $B$ are commutative algebras in $\cM$ and $f\colon A\rightarrow B$ is a morphism of commutative algebras. We define the \emph{space $\Pois(f, n)$ of $n$-shifted Poisson structures} on the morphism $f$ to be the fiber of the forgetful functor
\[\Arr(\balg_{\bP_{n+1}}(\cM))^\sim\longrightarrow \Arr(\bCAlg(\cM))^\sim\]
at $f\colon A\rightarrow B$.
\end{defn}

We can also consider the morphism $f\colon A\rightarrow B$ as a commutative algebra in $\Arr(\cM)$. In particular, Theorem \ref{thm:poissonpolyvectors} gives an efficient method of computing the space of Poisson structures on $f$.

We begin with the following general paradigm. Suppose $\cC$ is a monoidal $\infty$-category which admits sifted colimits and which are preserved by the tensor product. Denote by $\bimod{}{}(\cC)$ the $\infty$-category of triples $(A, B, M)$, where $A,B\in\balg(\cC)$ and $M$ is an $(A,B)$-bimodule. Similarly, let $\brmod(\cC)$ be the $\infty$-category of pairs $(B, M)$ where $B\in\balg(\cC)$ and $M$ is a right $B$-module. We have a functor
\[\balg(\cC)\rightarrow \brmod(\cC)\]
given by sending an associative algebra $B$ to the canonical right $B$-module $B$.

\begin{lm}
We have a Cartesian square of $\infty$-categories
\begin{equation}
\xymatrix{
\Arr(\balg(\cC)) \ar[r] \ar^{\mathrm{target}}[d] & \bimod{}{}(\cC) \ar[d] \\
\balg(\cC) \ar[r] & \brmod(\cC).
}
\label{eq:morphismbimodule}
\end{equation}
\label{lm:morphismbimodule}
\end{lm}
\begin{proof}
This is a slight variant of \cite[Corollary 4.8.5.6]{Lu}.

Let $\blmod_\cC$ be the $\infty$-category of $\cC$-module categories and $\blmod_\cC^{\ast}$ the $\infty$-category of pointed $\cC$-module categories, i.e.
\[\blmod_\cC^{\ast} = (\blmod_\cC)_{\cC/}.\]

By \cite[Theorem 4.8.4.1]{Lu} we have a Cartesian diagram of $\infty$-categories
\[
\xymatrix{
\bimod{}{}(\cC) \ar[d] \ar[r] & \Arr(\blmod_\cC) \ar[d] \\
\balg(\cC)\times \balg(\cC) \ar[r] & \blmod_\cC \times \blmod_\cC
}
\]
where $\bimod{}{}(\cC)\rightarrow \Arr(\blmod_\cC)$ sends a triple $(A, B, M)$ to the functor
\[-\otimes_A M\colon \brmod_A(\cC)\longrightarrow \brmod_B(\cC)\]
determined by $M$. Therefore, the pullback $P$ of the diagram \eqref{eq:morphismbimodule} fits into a Cartesian square
\[
\xymatrix{
P \ar[d] \ar[r] & \Arr(\blmod_{\cC}^\ast) \ar[d] \\
\balg(\cC)\times \balg(\cC) \ar[r] & \blmod_{\cC}^\ast \times \blmod_{\cC}^\ast
}
\]

By \cite[Theorem 4.8.5.5]{Lu} the functor $\balg(\cC)\rightarrow \blmod_\cC^{\ast}$ given by $A\mapsto \brmod_A(\cC)$ is fully faithful. Therefore, we have a Cartesian diagram of $\infty$-categories
\[
\xymatrix{
\Arr(\balg(\cC)) \ar[r] \ar[d] & \Arr(\blmod_\cC^{\ast}) \ar[d] \\
\balg(\cC)\times \balg(\cC) \ar[r] & \blmod_\cC^{\ast}\times \blmod_\cC^{\ast}
}
\]
and hence $P\cong \Arr(\balg(\cC))$.
\end{proof}

To prove the following theorem we are going to use the above lemma in the cases when $\cC=\balg_{\bP_n}(\cM)$ and $\cC=\bCAlg(\cM)$.

\begin{thm}
Let $f\colon A\rightarrow B$ be a morphism of commutative algebras in $\cM$ and $g\colon A\otimes B\rightarrow B$ its graph, i.e. $g(x\otimes y) = f(x)y$. Then we have a Cartesian square of spaces
\[
\xymatrix{
\Pois(f, n) \ar[r] \ar[d] & \Pois(A, n)\times \Pois(B, n) \ar^{\id\times \opp}[d] \\
\Cois(g, n) \ar[r] & \Pois(A\otimes B, n).
}
\]
\label{thm:graphcoisotropic}
\end{thm}
\begin{proof}
In the proof we will repeatedly use the statements of Poisson additivity given by Theorems \ref{thm:poissonadditivity} and \ref{thm:relpoissonadditivity}.

If $\cC$ is a symmetric monoidal $\infty$-category, we have a Cartesian square of $\infty$-categories
\begin{equation}
\xymatrix{
\bimod{}{}(\cC) \ar[r] \ar[d] & \balg(\cC) \times \balg(\cC) \ar[d] \\
\blmod(\cC) \ar[r] & \balg(\cC),
}
\label{eq:graph1}
\end{equation}
where the functor $\balg(\cC)\times\balg(\cC)\rightarrow \balg(\cC)$ is given by sending $(A, B)\mapsto A\otimes B^{op}$.

Let $F$ be the fiber of
\[\bimod{}{}(\balg_{\bP_n}(\cM))^\sim\rightarrow \bimod{}{}(\bCAlg(\cM))^\sim\]
at $(A, B, B)$. Taking the fiber of \eqref{eq:graph1} applied to the forgetful functor
\[(\cC_1=\balg_{\bP_n}(\cM))\longrightarrow (\cC_2=\bCAlg(\cM))\]
we get a Cartesian square of spaces
\[
\xymatrix{
F \ar[r] \ar[d] & \Pois(A, n)\times \Pois(B, n) \ar^{\id\times \opp}[d] \\
\Cois(g, n) \ar[r] & \Pois(A\otimes B, n),
}
\]
so we have to show that $F\cong \Pois(f, n)$.

Now consider the Cartesian square
\begin{equation}
\xymatrix{
\Arr(\balg(\cC)) \ar[r] \ar^{\mathrm{target}}[d] & \bimod{}{}(\cC) \ar[d] \\
\balg(\cC) \ar[r] & \brmod(\cC).
}
\label{eq:graph2}
\end{equation}
given by Lemma \ref{lm:morphismbimodule}. Taking the fiber of \eqref{eq:graph2} applied to the forgetful functor
\[(\cC_1=\balg_{\bP_n}(\cM))\longrightarrow (\cC_2=\bCAlg(\cM))\]
we get a Cartesian square of spaces
\[
\xymatrix{
\Pois(f, n) \ar[r] \ar[d] & F \ar[d] \\
\Pois(B, n) \ar[r] & \Cois(\id, n)
}
\]
The bottom morphism is induced by the functor
\[\balg(\cC)\rightarrow \blmod(\cC)\]
splitting the obvious forgetful functor
\[\blmod(\cC)\rightarrow \balg(\cC)\]
and hence $\Pois(B, n)\rightarrow \Cois(\id, n)$ is a weak equivalence by Proposition \ref{prop:idcoisotropic}. Therefore, $\Pois(f, n)\rightarrow F$ is also a weak equivalence which proves the claim.
\end{proof}
In particular, the fibers of the horizontal morphisms of the Cartesian diagram of Theorem \ref{thm:graphcoisotropic} are thus equivalent. It follows that given two Poisson structures $\pi_A$ and $\pi_B$ on $A$ and $B$ respectively, lifting a map of algebras $f : A \to B$ to a Poisson map is equivalent to giving a coisotropic structure on the graph $g : A \otimes B \to B$, where $A \otimes B$ is endowed with the $\bP_{n+1}$-structure $(\pi_A;-\pi_B)$.

\end{document}